\documentclass[11pt]{amsart}
\usepackage[english]{babel}
\usepackage[autostyle, english = american]{csquotes}
\MakeOuterQuote{"}
\usepackage[nointegrals]{wasysym}
\usepackage[all]{xy}
\usepackage{amssymb}
\usepackage{mathrsfs}
\usepackage{tensor}
\usepackage{verbatim}
\usepackage{graphicx}
\usepackage{scalerel}
\usepackage{xcolor}
\usepackage{todonotes}
\usepackage{hyperref}
\theoremstyle{plain}
\newtheorem{thm}{Theorem}[section]
\newtheorem{lem}[thm]{Lemma}
\newtheorem{prop}[thm]{Proposition}
\newtheorem{cor}[thm]{Corollary}
\newtheorem{fact}[thm]{Fact}

\newtheorem*{thmA}{Theorem A}
\newtheorem*{thmB}{Theorem B}
\newtheorem*{thmC}{Theorem C}
\newtheorem*{corE}{Corollary E}

\newtheorem*{quesF}{Question F}
\newtheorem*{quesD}{Question D}
\newtheorem*{corG}{Corollary G}
\newtheorem*{thmH}{Theorem H}

\theoremstyle{definition}

\theoremstyle{remark}
\newtheorem{rem}[thm]{Remark}
\numberwithin{equation}{section}

\newcommand{\iid}{\mathrm{id}}

\newcommand{\image}{\mathrm{im}}

\newcommand{\boh}{\mathbf{h}}
\newcommand{\gog}{\mathscr{G}}

\newcommand{\dbl}{[\![}
\newcommand{\dbr}{]\!]}

\newcommand{\Z}{\mathbb{Z}}
\newcommand{\R}{\mathbb{R}}
\newcommand{\C}{\mathbb{C}}
\newcommand{\Q}{\mathbb{Q}}

\newcommand{\euV}{\mathscr{V}}
\newcommand{\euE}{\mathscr{E}}

\newcommand{\euT}{\mathscr{T}}
\newcommand{\euH}{\mathscr{H}}

\newcommand{\eue}{\mathbf{e}}
\newcommand{\euf}{\mathbf{f}}

\newcommand{\dH}{\mathrm{dH}}
\newcommand{\dis}{\mathbf{dis}}
\newcommand{\QGdis}{{}_{\QG}\dis}
\newcommand{\QGmod}{{}_{\QG}\mathbf{mod}}
\newcommand{\Qvect}{{}_{\Q}\mathbf{vect}}

\newcommand{\nor}{|\! |}
\newcommand{\dd}{\mathrm{d}}

\newcommand{\eps}{\varepsilon}

\newcommand{\QG}{\Q[G]}
\newcommand{\CO}{\mathcal{CO}}
\newcommand{\caO}{\mathcal{O}}

\newcommand{\Bi}{\mathrm{Bi}}

\newcommand{\argu}{\hbox to 1.5ex{\hrulefill}}  

\newcommand{\FP}{\mathrm{FP}}
\newcommand{\ccd}{\mathrm{cd}}

\newcommand{\FF}{\mathrm{F}}

\newcommand{\ce}{:=}

\def\Q{{\mathbb{Q}}}

\def\Z{{\mathbb{Z}}}
\def\R{{\mathbb{R}}}
\def\C{{\mathbb{C}}}

\theoremstyle{definition}
\newtheorem{defn}[thm]{Definition}

\newtheorem{ex}[thm]{Example}

\usepackage{graphicx}
\usepackage{tikz}
\usepackage{tikz-cd}

\begin{document}
\begin{abstract} 
It is shown that a Stallings--Swan theorem holds in a
totally disconnected locally compact (= t.d.l.c.) context (cf.~Thm.~B). More precisely, a compactly generated 
$\CO$-bounded t.d.l.c.~group $G$ of rational discrete cohomological
dimension less than or equal to $1$ 
must be isomorphic to the fundamental group of a finite graph of
profinite groups. This result generalises Dunwoody's rational version of the classical Stallings--Swan theorem to t.d.l.c.~groups. The proof of Theorem~B 
is based on the fact that a compactly generated unimodular t.d.l.c.~group
with rational discrete cohomological dimension $1$ has necessarily non-positive Euler--Poincaré characteristic
(cf.~Thm.~H).
\end{abstract}
\date{\today}

\title[Unimodular t.d.l.c.~groups of cohomological dimension one]{Unimodular totally disconnected\\ locally compact groups of rational discrete cohomological dimension one}
\author{Ilaria Castellano}
\address{Heinrich Heine Universit\"at D\"usseldorf,
Mathematisch-Naturwissenschaftliche Fakult\"at, Universit\"atsstra\ss{}e 1, 40225 D\"usseldorf, Germany}
\email{ilaria.castellano@hhu.de}
\author{Bianca Marchionna}
\address{Heidelberg University \\
Faculty of Mathematics and Computer Science, Im Neuenheimer Feld 205, 69120 Heidelberg, Germany}
\email{bmarchionna@mathi.uni-heidelberg.de}
\author{Thomas Weigel}
\address{Dipartimento di Matematica e Applicazioni, Universit\`a degli Studi di
Milano-Bicocca, Via Roberto Cozzi no. 55, I-20125 Milano, Italy}
\email{thomas.weigel@unimib.it}
\maketitle
\section{Introduction}
\label{s:intro} 
The classical Stallings--Swan theorem states that a (discrete) group $G$ is free if, and only if, $\ccd_{\Z}(G)\leq 1$. {It was firstly proved by J.~R.~Stallings~\cite{stall:end} for finitely generated groups, and later shown to hold without any additional hypothesis
by R.~G.~Swan~\cite{swan:cone}, who actually proved that every torsion-free group with cohomological dimension at most one over a non-trivial ring with unit is free. 
The Stallings--Swan theorem is considered as a milestone in the cohomology theory of discrete groups. 
Some time later, M.~Dunwoody~\cite{dunw:access} provided the following further extension.}
\begin{thmA}[M.~Dunwoody, 1979]\label{thmA}
Let $G$ be a group and $R$ be a commutative ring with $1\neq0$. Then
the following are equivalent:
\begin{itemize}
\item[(i)] $\ccd_R(G)\leq 1$;
\item[(ii)] there exists a connected graph $\Lambda$ and a graph of finite
groups $(\gog,\Lambda)$ based on $\Lambda$ such that $|\gog_v|$ is invertible in $R$ for all $v\in\euV(\Lambda)$
and
$G\simeq\pi_1(\gog,\Lambda)$. 
\end{itemize}
\end{thmA}
The main goal of this note is to establish an analogue of Theorem~\hyperref[thmA]{A} for
t.d.l.c.~groups. Unfortunately, we were not able to prove it in full generality for this class of groups, but we had to assume further finiteness conditions.

The first finiteness condition on the t.d.l.c.~group we require is compact generation. Recall that a t.d.l.c.~group $G$ is said to be \emph{compactly generated} if there exists a compact set $\Sigma\subseteq G$
such that every element of $G$ can be written as a finite product of elements of $\Sigma$. E.g., a discrete group $G$
is compactly generated if, and only if, it is finitely generated. By van Dantzig's theorem (cf.~\cite{vD:CO}),
the set $\CO(G)\ce\{\caO\mid \caO \text{ compact open subgroup of }G\}$ is non-empty, 
and it is well known that every group $\caO\in\CO(G)$ is profinite
(cf.~\cite[\S I.1, Proposition~0]{ser:gal}). Note that $G$ is compactly generated if, and only if, for every $\caO\in\CO(G)$ there exists a finite set $S\subseteq G\setminus \caO$ such that $\caO\sqcup S$ generates $G$.

The second finiteness condition we consider is that the t.d.l.c.~group $G$ is \emph{$\CO$-bounded} (cf.~Section~\ref{sus:CObound}). Namely, for a fixed  left-invariant Haar measure $\mu$ on $G$ one requires that there exists a positive integer $c\in\Z_{\geq 1}$ 
such that every compact open subgroup $\caO$ of $G$ satisfies $\mu(\caO)\leq c$.
Every $\CO$-bounded t.d.l.c.~group must be unimodular (cf.~Proposition~\ref{prop:COunimod}).

In Section~\ref{s:acc} we show that the following analogue of Theorem~\hyperref[thmA]{A} holds for t.d.l.c.~groups.
\begin{thmB}[\protect{cf.~Theorem~\ref{thm: firstRes}~and~Corollary~\ref{cor:acc}}]\label{thmB}
Let $G$ be a compactly generated t.d.l.c.~group.
Then the following are equivalent:
\begin{itemize}
  \item[(i)] $\ccd_\Q(G)\leq 1$ and $G$ is $\CO$-bounded;
  \item[(ii)] $G$ is unimodular and there exists a finite connected graph $\Lambda$ and a graph of profinite groups $(\gog, \Lambda)$
such that $G\simeq \pi_1(\gog,\Lambda)$.
\end{itemize}
Here $\ccd_\Q(\argu)$ denotes the rational discrete cohomological dimension for a t.d.l.c.~group as introduced in \cite{cw:qrat}.
\end{thmB}
Recall that a \emph{graph of t.d.l.c.~groups} $(\gog,\Lambda)$ consists of a connected graph~$\Lambda$, t.d.l.c.~groups $\gog_v$, $v\in\euV(\Lambda)$, and
$\gog_\eue$, $\eue\in\euE(\Lambda)$, and 
injective open continuous group homomorphisms $\alpha_\eue\colon\gog_\eue\to\gog_{t(\eue)}$ for all $\eue\in\euE(\Lambda)$. 
There is a canonical way to assign a t.d.l.c.~group topology to the Bass--Serre fundamental group of a graph of t.d.l.c.~groups (cf.~\S\ref{ss:gog}).
A compactly generated t.d.l.c.~group $G$ is said to be {\em accessible} if it is topologically isomorphic to the Bass--Serre fundamental group of a finite graph of t.d.l.c.~groups $(\gog,\Lambda)$ with compact edge-groups and whose vertex-groups have at most one end (cf.~for instance~\cite[\S 3.3]{km:rough}). In case $\ccd_\Q(G)\leq 1$ then $G$ is accessible if, and only if, it is topologically isomorphic to the fundamental group of a finite graph of profinite groups (cf.~Fact~\ref{fact:one end}).
It should be mentioned that Theorem~\hyperref[thmB]{B} is not the first result of this kind. Indeed, in \cite{IC:cone},
the first named author established an analogue of  Theorem~\hyperref[thmA]{A} based on an accessibility result 
of Y.~Cornulier (cf.~\cite[Theorem 4.H.1]{cor:qic}). The result can be summarised as follows.
\begin{thmC}[\protect{cf.~\cite[Theorem~B]{IC:cone}}]\label{thmC}
For a compactly generated t.d.l.c.~group $G$ the following are equivalent:
\begin{itemize}
\item[(i)] $\ccd_\Q(G)\leq 1$ and $G$ is compactly presented;
\item[(ii)] there exists a finite connected graph $\Lambda$ and a graph of profinite groups
$(\gog,\Lambda)$ such that $G\simeq\pi_1(\gog,\Lambda)$.
\end{itemize}
\end{thmC}
An affirmative answer to the following question, which was raised by the referee, would provide a t.d.l.c.~analogue to Theorem~\hyperref[thmA]{A} in the case of compact generation. 
\begin{quesD}\label{quesD}
Is a t.d.l.c.~group of rational discrete cohomological dimension~$1$ \emph{coherent} (i.e., every compactly generated closed subgroup is compactly presented)?
\end{quesD}
Theorem~\hyperref[thmB]{B} and Theorem~\hyperref[thmC]{C} admit the following interpretation.
A compactly generated t.d.l.c.~group $G$ satisfying $\ccd_{\Q}(G)\leq 1$ and
an extra finiteness condition is accessible in the sense above. This extra finiteness condition is compact presentability -- which is equivalent
to being of type $\FF_2$ (cf.~\cite[Proposition~3.4]{ccc:finite}) -- or being $\CO$-bounded.
At first sight, these two conditions
seem unrelated. However, for example, it is well-known that a unimodular t.d.l.c.~group $G$ admitting a cocompact topological model
is necessarily $\CO$-bounded (cf.~\cite[Remark~1.11]{RT2017}). Hence another consequence of Theorems~\hyperref[thmB]{B}~and~\hyperref[thmC]{C} is the following.
\begin{corE}\label{corE}
Let $G$ be a compactly generated t.d.l.c.~group with \mbox{$\ccd_{\Q}(G)\leq 1$.}
Then $G$ is {$\CO$-bounded}  if, and only if, it is compactly presented and unimodular.
\end{corE}
In \cite[Theorem~1]{lin:acc}, P.~A.~Linnell proved that a finitely generated group whose finite subgroups have bounded order is accessible.
The following question asks for a t.d.l.c.~analogue of P.~A.~Linnell's result.
\begin{quesF}
    Is a compactly generated $\CO$-bounded t.d.l.c.~group accessible?
\end{quesF}
Another consequence of Theorem~\hyperref[thmB]{B} is the following partial answer to Question~\hyperref[quesF]{F}.
\begin{corG}\label{corG}
Let $G=\pi_1(\gog, \Lambda)$ be the t.d.l.c.~fundamental group of a finite graph $(\gog, \Lambda)$ of $\CO$-bounded t.d.l.c.~groups with profinite edge-groups and $\ccd_\Q(G)\leq 1$. If $G$ is compactly generated, then $G$ is accessible.
\end{corG}
\begin{proof}
Each vertex-group $\gog_v$ of $(\gog,\Lambda)$ is an open subgroup of $G$. So if $G$ is compactly generated, 
every $\gog_v$ is a compactly generated t.d.l.c.~group satisfying  $\ccd_\Q(\gog_v)\leq 1$ (cf.~\cite[Proposition~4.1]{IC:cone} and \cite[Proposition~3.7(c)]{cw:qrat}). As every t.d.l.c.~fundamental group of a finite graph of accessible t.d.l.c.~groups with compact edge-groups is accessible (cf.~Proposition~\ref{prop: AccFact}), the claim directly follows from Theorem~\hyperref[thmB]{B}.
\end{proof}
The proof of Theorem~\hyperref[thmB]{B} is based on the following important result about the Euler--Poincaré characteristic $\tilde{\chi}_G$ of a t.d.l.c.~group $G$. This invariant, introduced in \cite{ccw:euler}, is defined only for unimodular t.d.l.c.~groups of type $\FP$ over $\Q$.
\begin{thmH}[\protect{cf.~Theorem~\ref{thm:chi}}]\label{thmH}
Let $G$ be a compactly generated unimodular t.d.l.c.~group with $\ccd_\Q(G)=1$. Then the Euler--Poincaré characteristic $\tilde{\chi}_G$ of $G$ satisfies $\tilde{\chi}_G\leq 0$. Moreover, $\tilde{\chi}_G=0$ if and only if $G\simeq X\ast_UY$ for some compact open subgroups $X,Y,U\leq G$ satisfying $|X:U|=|Y:U|=2$.
\end{thmH}
The proof of Theorem~\hyperref[thmH]{H} is inspired by a technique used by P.~A.~Linnell~\cite{lin:acc} and contributes to the general understanding of the Euler--Poincaré characteristic of a t.d.l.c.~group. So far the conclusion of 
Theorem~\hyperref[thmH]{H} was only known for non-compact unimodular fundamental groups of finite graphs of profinite groups (cf.~\cite[Corollary~C]{ccw:euler}). Note that Theorem~\hyperref[thmA]{A} implies that the classical rational Euler--Poincaré characteristic $\chi_G$ is non-positive if the group $G$ is finitely generated and $\ccd_\Q(G)=1$. Therefore, Theorem~\hyperref[thmH]{H} can be considered as generalisation of a well-known result for finitely generated virtually free groups. 

We conclude this preliminary discussion with the following observation.
Since the family of finite subgroups of a discrete group has a minimal element, for a finitely generated group $G$ being accessible is equivalent to the following property: there exists $k(G)<\infty$ such that $|\euE(\Lambda)|\leq k(G)$ for every minimal graph of groups $(\gog,\Lambda)$ with finite edge-groups such that $G\simeq\pi_1(\gog,\Lambda)$ (\cite[Lemma~6.7]{sw77}).
For example, in the context of pro-$p$ groups where an analogue of the Stallings' decomposition theorem is missing, the latter property is used to define the notion of accessibility (cf.~\cite{cz23,cz22,wilk19}). For a compactly generated t.d.l.c.~group we do not know whether being accessible is equivalent to the existence of a uniform bound $k(G)$ for the number of edges appearing in a decomposition as fundamental group of a graph of t.d.l.c.~groups with compact edge-groups. However, for compactly generated unimodular t.d.l.c.~groups of rational discrete cohomological dimension $1$ we prove the latter equivalence in Corollary~\ref{cor:inaccuni}.
\subsection*{Organisation}
We start with an introduction to the rational discrete cohomology for t.d.l.c.~groups, which is the only cohomology theory used in this paper, see Section~\ref{s:Rat}.
In Section~\ref{s:W*} we deal with von Neumann algebras, and recall the definition of the $W^*$-algebra $W(G,\caO)$ associated to the Hecke pair $(G, \caO)$, where $G$ is a t.d.l.c.~group and $\caO$ is a compact open subgroup. On such a $W^*$-algebra we prove the existence of a well-defined trace map $\tau\colon W(G,\caO)\to \C$, that is playing a central role throughout the paper. A few technical properties of $\tau$ are then provided in support of the proof of Proposition~\ref{prop: Approx}, which is one of the key results of the paper. Section~\ref{s:euchar} is indeed mainly devoted to the proof of Proposition~\ref{prop: Approx}, which provides a suitable lower bound for the Hattori--Stallings rank of the augmentation left $\QG$-module $N^G_{\caO}$. The module $N^G_{\caO}$ is defined as the kernel of the augmentation map $\Q[G/\caO]\to\Q$, which turns out to be a finitely generated projective discrete $\QG$-module exactly when $G$ is compactly generated and $\ccd_\Q(G)\leq1$. Therefore, $N^G_\caO$ provides a projection in the $W^*$-algebra of matrices $M_n(W(G,\caO))$ whose trace is then estimated.
A notable consequence of Proposition~\ref{prop: Approx} is Theorem~\ref{thm:chi}. It establishes the non-positivity of the Euler--Poincaré characteristic of a compactly generated unimodular t.d.l.c.~group of rational discrete cohomological dimension $1$.

Summing up, Sections~\ref{s:W*}~and~\ref{s:euchar} consist of the technical part of the paper, which is the crux of our main results. Section~\ref{s:acc} moves the focus onto accessibility. After some preliminary definitions, we delineate two possible scenarios for compactly generated t.d.l.c.~groups which are inaccessible (cf.~Proposition~\ref{prop:inacc}). Using Theorem~\ref{thm:chi}, we may exclude one of them provided the t.d.l.c.~group $G$ satisfies $\ccd_\Q(G)=1$.
Section~\ref{s:acc} culminates with Theorem~\ref{thm: firstRes}, which proves that a compactly generated unimodular t.d.l.c.~group $G$ is accessible whenever 
$\ccd_\Q(G)=1$ and $G$ is $\CO$-bounded, i.e., the compact open subgroups in $G$ have a finite uniform bound on their Haar measure. Examples of $\CO$-bounded groups are unimodular t.d.l.c.~groups that admit a topological model on which they act with finitely many orbits on the $0$-cells, cf.~Example~\ref{ex: prelThm}.
\subsection*{Conventions}
In this paper, every isomorphism between topological groups is always meant to be continuous and with continuous inverse.

\section{Rational discrete cohomology: a summary}\label{s:Rat}
This section summarises the main definitions and properties of the rational discrete cohomology for t.d.l.c.~groups (cf.~\cite{cw:qrat}) that will be used throughout the paper.

\smallskip
Let $G$ be a t.d.l.c.~group. A \emph{discrete left $\QG$-module} $M$ is an abstract left $\Q[G]$-module equipped with the discrete topology such that the $G$-action $\cdot\colon G\times M\longrightarrow M$ is a continuous map (here $G\times M$ is endowed with the product topology). 
\begin{fact}\label{fact:openstab}
Let $G$ be t.d.l.c.~group. A left $\QG$-module $M$ is discrete if, and only if, for every $m\in M$ the group $\mathrm{stab}_G(m)=\{g\in G\mid g\cdot m=m\}$ is open in $G$.
\end{fact}
The category $\QGdis$ is the full subcategory of $\QGmod$ whose objects are discrete left $\QG$-modules (here $\QGmod$ denotes the category of abstract left $\QG$-modules). By \cite[Fact~2.2~and~Proposition~3.2]{cw:qrat}, $\QGdis$ is an abelian category with enough injectives and enough projectives. 
In particular, for every $\caO\in\CO(G)$, the \emph{transitive left $\QG$-permutation module} $\Q[G/\caO]$ is projective in $\QGdis$ (cf.~\cite[Proposition~3.2]{cw:qrat}). By Fact~\ref{fact:openstab}, every $M\in\mbox{ob}(\QGdis)$ admits a collection of compact open subgroups $\{\caO_i\}_{i\in I}$ of $G$ together with an epimorphism 
\begin{equation}\label{eq:epi}
\pi_M\colon \coprod_{i\in I}\Q[G/\caO_i]\twoheadrightarrow M.
\end{equation}
In particular $M$ is projective if, and only if, it is a direct summand of $\coprod_{i\in I}\Q[G/\caO_i]$ (cf.~\cite[Corollary~3.3]{cw:qrat}).
Moreover, $M$ is said to be \emph{finitely generated} in $\QGdis$ if the set $I$ can be chosen finite. One easily checks that this is equivalent to the existence of finitely many elements $m_1,...,m_n$ in $M$ satisfying $M=\sum_{i=1}^n \Q[G]\cdot m_i$. 

\smallskip

According to \cite[\S3]{cw:qrat}, for every $n\geq 0$ the \emph{$n^{th}$ rational discrete cohomology functor} of $G$, denoted by 
\begin{equation}
\dH^n(G,\argu)\colon\QGdis \longrightarrow \Qvect,
\end{equation}
is defined as the $n^{th}$ right-derived functor of the fixed-point functor $(\argu)^G$ from $\QGdis$ to $\Qvect$.
The \emph{rational discrete cohomological dimension} of $G$  is given by 
\begin{equation}\label{eq: cd}
\ccd_\Q(G)\ce\sup\{n\in\Z_{\geq 0}\mid\dH^n(G,\argu)\neq 0\}.
\end{equation}
(cf.~\cite[p.~115]{cw:qrat}). By \cite[Proposition~3.7]{cw:qrat}, $\ccd_\Q(G)=0$ if and only if $G$ is compact (i.e., profinite), and $\ccd_\Q(H)\leq \ccd_\Q(G)$ whenever $H$ is a closed subgroup of $G$. By~\cite[Lemma~3.6]{cw:qrat}, $\ccd_\Q(G)\leq 1$ if and only if the trivial discrete $\QG$-module~$\Q$ admits a projective resolution of length at most~$1$ in $\QGdis$.

\smallskip

Unless $G$ is discrete, the group algebra $\QG$ is not an object of $\QGdis$ (cf.~Fact~\ref{fact:openstab}). A suitable substitute of the group algebra has been introduced in \cite[\S4.2]{cw:qrat}: the \emph{rational discrete standard $\QG$-bimodule} $\Bi(G)$. It is defined as a direct limit of $\{\Q[G/\caO]\mid \caO\in \CO(G)\}$. If $G$ is discrete, $\Bi(G)$ is naturally isomorphic to $\QG$. By \cite[Remark~4.3]{cw:qrat}, when $G$ is unimodular, $\Bi(G)$ is (non-canonically) isomorphic to the associative convolution algebra $C_c(G,\Q)$ of all continuous functions with compact support from $G$ to $\Q$, where $\Q$ has the discrete topology. The reader may find it convenient to keep this isomorphism in mind while reading Section~\ref{s:W*}.

In analogy to the discrete case, one can use the rational discrete standard bimodule to give a cohomological characterisation of a compactly generated t.d.l.c.~group $G$ which has more than one end. Namely, a compactly generated group $G$ has \emph{more than one end} if, and only if, $\dH^1(G,\Bi(G))\neq0$ (cf.~\cite[Theorem~$\textup{A}^*$]{IC:cone}). Moreover, $G$ has \textit{zero ends} if and only if it is compact.

\smallskip

According to \cite[\S3.6~and~\S4.5]{cw:qrat}, a discrete left $\QG$-module $M$ is of \emph{type $\FP_n$} ($n\geq 0$) if it admits a partial projective resolution in $\QGdis$
\begin{equation}\label{eq:finproj}
\xymatrix{P_n\ar[r]&\ldots\ar[r]&P_1\ar[r]&P_0\ar[r]&M\ar[r]&0}
\end{equation}
 with $P_0,\ldots,P_n$ finitely generated. The  discrete left $\QG$-module $M$ is  of \emph{type $\FP_\infty$} if it is of type $\FP_n$ for all $n\geq 0$, and  $M$ is of \emph{type $\FP$} if there exists a projective resolution in $\QGdis$
\begin{equation}\label{eq:finprojFP}
\xymatrix{0\ar[r]&P_n\ar[r]&\ldots\ar[r]&P_1\ar[r]&P_0\ar[r]&M\ar[r]&0}
\end{equation}
 which has finite length $n\geq0$ and every $P_i$ being finitely generated. A projective resolution with the latter two properties is called {\em finite}. 
The t.d.l.c.~group $G$ is of \emph{type $\FP_n$} (resp.~of \emph{type $\FP$}) if $\Q$ is of type $\FP_n$ (resp.~of type $\FP$) as trivial discrete left $\QG$-module. E.g., compact generation is equivalent to being of type $\FP_1$ (cf.~\cite[Proposition~5.3]{cw:qrat}).
By \cite[Ch.~VIII, Proposition~6.1]{bro:coho}, $G$ is of \emph{type $\FP$} if, and only if, it is of type $\FP_\infty$ and $\ccd_\Q(G)<+\infty$.

\smallskip

Similarly to the discrete case, one introduces the notion of \emph{t.d.l.c.~group of type $\FF_n$} ($0\leq n\leq +\infty$) and \emph{t.d.l.c.~group of type $\FF$}.
A {\it discrete $G$-CW-complex} is a $G$-CW-complex $X$ such that the action of $G$ on $X$ is continuous and by cell-permuting homeomorphisms\footnote{As usual, an element $g\in G$ fixing a cell $\sigma$ setwise can be assumed to fix $\sigma$ also pointwise.} (cf.~\cite[p.~5]{mv}). A discrete $G$-CW-complex $X$ is said to be {\em proper} if the cell stabilisers are compact.
As pointed out in \cite[\S6.2]{cw:qrat}, $G$ always admits an {\it $\underbar{E}_{\CO}(G)$-space}, i.e.~a proper discrete $G$-CW-complex $X$ which is contractible and such that, for all $\caO\in\CO(G)$, the fixed-point set $X^{\caO}$ is non-empty and contractible. According to \cite{ccc:finite}, a t.d.l.c.~group $G$ is of \emph{type $\FF_n$} (with $0\leq n<+\infty$) if there exists a contractible
proper discrete $G$-CW-complex $X$  such that $G$ acts on the
$n$-skeleton of $X$ with finitely many orbits. In particular, $G$ is of \emph{type $\FF$} if there exists a finite-dimensional contractible proper discrete
$G$-CW-complex $X$ such that, for every $n\geq 0$, $G$ acts on the
$n$-skeleton of $X$ with finitely many orbits.
\section{The $W^*$-algebra of a Hecke pair}\label{s:W*}
\subsection{$C^*$-algebras}
An {\it algebra of operators} $A$ is a subalgebra of the algebra $\mathcal{B}(H)$ of all bounded operators on a Hilbert space $H$. The algebra $A$ is said to be {\it self-adjoint} if it is closed under the adjoint operation $(\argu)^*\colon\mathcal{B}(H)\to\mathcal{B}(H)$, i.e., $A=A^*$.
A uniformly closed self-adjoint algebra of operators is called a {\it $C^*$-algebra}. By the Gelfand-Naimark theorem, $C^*$-algebras can be defined in a completely abstract (but equivalent) manner. Namely, a \emph{$C^*$-algebra} is a Banach algebra $A$ with an involution $(\argu)^*$ (i.e., a conjugate-linear self-map of $A$ such that $x^{**}=x$ and $(xy)^*=y^*x^*$ for all $x,y\in A$) which satisfies $\nor x^*x\nor=\nor x\nor^2$ for all $x\in A.$
Hence, one may study $C^*$-algebras without paying attention to any particular representation.
\subsection{$W^*$-algebras}\label{ss:w*}
A {\it $W^*$-algebra} (or {\it von Neumann algebra}) is a weakly closed self-adjoint algebra of operators on a Hilbert space $H$. Let $A$ be an algebra of operators which contains the identity operator and is closed under taking adjoints. By \cite[Theorem 1.2.1]{arveson}, the weak closure of $A$ is equal to the bicommutant\footnote{
For every subset $S\subseteq A$, let $S'\ce\{a\in A: as=sa, \forall\, s\in S\}$ be the \textit{commutant of $S$}. Then $S''\ce(S')'$ is called the \textit{bicommutant of $S$}.} $A''$ of $A$ and it is called the {\it $W^*$-algebra generated by $A$}.

 Let $A$ be a $W^*$-algebra acting on the complex Hilbert space $H$. Denote by $M_n(A)$ the set of $n\times n$ matrices with entries in $A$. The algebra $M_n(A)$ acts on the $n$-fold direct sum $H^n=H\oplus\cdots\oplus H$ through the usual matrix action on the column vectors. Therefore, $M_n(A)$ can be regarded as a self-adjoint subalgebra of $\mathcal{B}(H^n)$ where, for every matrix $M=[m_{jk}]\in M_n(A)$, the matrix $M^\ast=[n_{jk}]$ is defined by
$n_{jk}:=m^*_{kj}.$
 In order to show that $M_n(A)$ is a $W^*$-algebra, it suffices to check that the double commutant $M_n(A)''$ coincides with $M_n(A)$. Since the commutant $M_n(A)'$ is $A'\cdot I_n$ (where $I_n$ is the identity matrix and $A^\prime$ is the commutant of $A$ acting on $H$), it is easily verified that
$$M_n(A)''=M_n(A'')=M_n(A).$$
\begin{defn}
A (finite) {\em trace} on a $W^*$-algebra $A$ is a $\C$-linear function $\tau\colon A\to\C$ satisfying
$\tau(ab)=\tau(ba)$ for all $a,b\in A$. A trace $\tau$ is {\it positive} if $\tau(a^*a)\geq 0$ for every $a\in A$. It is said to be {\em faithful} if for every $a\in A$ one has $\tau(a^*a)=0$ if, and only if, $a^*a=0$.
\end{defn}
If $A$ is a $W^*$-algebra with trace $\tau$, then $M_n(A)$ is a $W^*$-algebra with trace
\begin{equation}
    (\tau\otimes\mathrm{id}_n)(M)\ce\tau(m_{11})+\ldots+\tau(m_{nn}),
\end{equation}
for any matrix $M=[m_{ij}]\in M_n(A)$.
\begin{rem} Despite the $C^*$-case, no intrinsic axioms are known for $W^*$-algebras. I.~Kaplanski \cite{kap:proj} proposed an algebraic {generalisation} of $W^*$-algebras based on the assumption of least upper bounds in the poset of projections (= self-adjoint idempotents) of the operator algebra: the so called {\em $AW^*$-algebras}. Every $W^*$-algebra is an $AW^*$-algebra.
\end{rem}
\subsection{The algebra $\mathcal H(G,\mathcal O)_\Q$}
Let $G$ be a t.d.l.c.~group and $\CO(G)=
\{\caO\mid \caO\ \text{compact open subgroup of $G$}\}$. Denote by $C_c(G,\Q)$ the space of  continuous functions with compact support from $G$ to $\Q$, where $\Q$ carries the discrete topology. It has a natural structure of discrete $\Q[G]$-bimodule given by the commuting actions
\begin{equation}
((g.f). h))(x)=f(g^{-1}(xh^{-1}))=f((g^{-1}x)h^{-1})=(g.(f. h))(x),
\end{equation}
for all $g,h,x\in G$ and $f\in C_c(G,\Q)$. Given $S\subseteq G$, let $I_S(\argu)\colon G\to\Q$ be the map defined by $I_S(x)=1$ if $x\in S$ and $I_S(x)=0$ otherwise.
Set
\begin{equation}\label{eq:dbr}
    [g\dbr_\caO\ce I_{g\caO}\quad\text{and}\quad \dbl g\dbr_\caO\ce I_{\caO g\caO},\quad \forall\, \caO\in\CO(G).
\end{equation}
Note that
$C_c(G,\Q)=\mathrm{span}_\Q\{{[g\dbr_\caO\mid g\in G,\,\mathcal O\in\mathcal{CO}(G)}\}$. Moreover, for every $\caO\in\CO(G)$,
\begin{equation}\label{eq: cgo}
C_c(G,\Q)^\caO\ce\{f\in C_c(G,\Q)\mid \forall\, w\in \caO, f.\omega=f\}={\mathrm{span}_\Q\{[g\dbr_\caO\mid g\in G\}}.
\end{equation}
and
\begin{align}\label{eq:hgo}
\mathcal{H}(G,\mathcal O)_\Q \ce& \, \{f\in C_c(G, \Q)\mid\forall \omega\in \mathcal O\colon \omega .f=f=f.\omega\} \\
=&\, {\mathrm{span}_\Q\{\dbl g\dbr_\caO\mid g\in G\}}.\nonumber
\end{align}
Let $\mu_{\mathcal O}$ be the Haar measure on $G$ such that $\mu_{\mathcal O}(\mathcal O)=1$. We omit the subscript ${\_\,}_\caO$ whenever the compact open subgroup $\caO$ is clear from the context.

The space $\mathcal{H}(G,\mathcal O)_\Q$ becomes an algebra when equipped with the convolution product given by
\begin{equation}
(f_1\ast_{\mu_{\caO}} f_2)(x)\ce\int_G f_1(w)f_2(w^{-1}x)\dd\mu_{\mathcal O}(w),\quad\forall f_1,f_2\in\mathcal{H}(G,\mathcal O),
\end{equation}
and it is called the {\it Hecke $\Q$-algebra of the Hecke pair $(G,\caO)$} (cf.~\cite[$\S$3.5]{ccw:euler}).
\smallskip

By \cite[Fact~3.2(a)]{ccw:euler}, the discrete left $\Q[G]$-module $C_c(G,\Q)^\caO$ is isomorphic to the left discrete $\QG$-permutation module $\Q[G/\caO]$, for every $\caO\in\CO(G)$. In particular, one has the following.

\begin{fact}[\protect{\cite[Proposition~3.7(b)]{ccw:euler}}]\label{fact: EndoH(G,O)nxn}
Let $G$ be a t.d.l.c.~group with a compact open subgroup $\caO$. 
Let $\phi_\ast\colon\mathcal{H}(G,\caO)_\Q^{op}\to \mathrm{End}_G(C_c(G,\Q)^\caO)$ denote the isomorphism of \protect{\cite[Proposition~3.7(b)]{ccw:euler}}. Then:
\begin{itemize}
 \item[(a)] there is an isomorphism of $\Q$-algebras  $$\Phi_*\colon M_n\Big(\mathcal{H}(G,\caO)_\Q^{op}\Big)\to \mathrm{End}_G\Big ((C_c(G,\Q)^\caO)^n\Big )$$
 defined, for every $f=[f_{ij}]\in M_n(\mathcal{H}(G,\caO)_\Q^{op})$, as
$$\Phi_*(f)\Big(\, [0,\ldots,\stackrel{\text{j-th}}{[g\dbr},\ldots,0\Big ]^t\,\Big)\ce\Big [\phi_\ast(f_{1j})([g\dbr),\ldots,\phi_\ast(f_{nj})([g\dbr) \Big ]^t, \ \forall\,j=1,\ldots, n;$$
\item[(b)] the inverse  $\Psi_\ast\colon\mathrm{End}_G\Big ((C_c(G,\Q)^\caO)^n\Big ) \to M_n\Big(\mathcal{H}(G,\caO)_\Q^{op}\Big)$ of $\Phi_*$ is 
$$(\Psi_\ast(\alpha))_{ij}\ce(\phi_\ast)^{-1}(\alpha_{ij}),\quad \forall 1\leq i,j\leq n.$$
Here for each $\alpha\in \mathrm{End}_G\Big ((C_c(G,\Q)^\caO)^n\Big )$ {and $i,j\in\{1,\ldots,n\}$, the map} $\alpha_{ij}\in \mathrm{End}_G(C_c(G,\Q)^\caO)$ is defined as
$$\alpha_{ij}([g\dbr)\ce pr_i\Big (\alpha\Big(\Big [0,...,\stackrel{\text{j-th}}{[g\dbr},...,0\Big ]\Big )\Big ),\qquad {\forall g\in G,}
$$
where $pr_i:(C_c(G,\Q)^\caO)^n\longrightarrow C_c(G,\Q)^\caO$ projects on the $i$-th component.
\end{itemize}
\end{fact}
\subsection{$\mathcal{H}(G,\mathcal O)_\C$ as operator algebra} 
In the following, $G$ is a unimodular compactly generated t.d.l.c.~group and, for every compact open subgroup $\caO\leq G$, let $\mathcal{H}(G,\caO)_\C\ce\mathcal{H}(G,\caO)_\Q\otimes_\Q \C$. Denoting by $\mu$ a Haar measure on~$G$, let
\begin{equation}\label{eq:L2GC}
    L^2(G,\C)\ce\Bigg \{f\colon G\to \C\ \mbox{measurable} \,\,\Bigg|\,\, \int_G|f(w)|^2 \dd\mu(w)<\infty\Bigg \}
\end{equation} 
be equipped with the inner product $\langle f_1,f_2\rangle=\int_G f_1(w)\overline{f_2(w)}\dd\mu(w),$ for all $f_1,f_2\in L^2(G,\C)$, and let $\nor\argu\nor_2\colon L^2(G,\C)\to\R_0^+$ be the Hilbert norm associated to $\langle \argu,\argu\rangle$.
In~\eqref{eq:L2GC}, $\C$ is endowed with the usual Borel $\sigma$-algebra.

Since $G$ acts continuously on $L^2(G,\C)$, for every compact open subgroup~$\caO\leq G$ the $\C$-subspace
\begin{equation}\label{eq:L2O}
L^2(G,\C)^{\mathcal O}\ce\{f\in L^2(G,\C)\mid f.\omega=f,\,\forall\,\omega\in\caO\}
\end{equation}
 is closed in $L^2(G,\C)$ and therefore it is a Hilbert subspace.  
Since $G$ is compactly generated, there is a finite set $S\subseteq G$ such that $G=\langle S\rangle\cdot \caO$ (cf.~\cite[Lemma~2]{mol}) and~$G/\caO$ is countable.
Given a set $\mathcal R$ of representatives for $G/\caO$, the set $\{[g\dbr\mid g\in\mathcal R\}$ (cf.~\eqref{eq:dbr}) is a countable orthonormal (Hilbert) basis of the Hilbert space $L^2(G,\C)^\caO$ (cf.~\cite[Fact~3.3]{ccw:euler}). 

Denote by $\mathcal{B}(G,\mathcal O)$ the $C^*$-algebra of bounded linear operators on~$L^2(G,\C)^{\mathcal O}$ which commute with the left $G$-action. The space $\mathcal{B}(G,\mathcal O)$ comes with the standard operator norm $\nor\argu\nor\colon\mathcal{B}(G,\mathcal O)\to\mathbb{R}_0^+$
and with the adjoint map $(\argu)^*:\mathcal{B}(G,\mathcal O)^{\text{op}}\to \mathcal{B}(G,\caO)$. 
 As shown in {\cite[Proposition~3.7]{ccw:euler}}, $\mathcal{H}(G,\caO)_\C$ admits an action on the Hilbert space $L^2(G,\C)^\caO$ and the convolution produces the following canonical injective homomorphism of algebras:
\begin{equation}\label{eq:phi}
    \phi\colon\mathcal{H}(G,\mathcal O)_\C^{op}\to\mathcal{B}(G,\mathcal O),\quad (\phi(f))([g\dbr)=[g\dbr\ast_{\mu_\caO}f,\quad\forall g\in G.
\end{equation}
{Here $\mu_\caO$ denotes the left Haar measure on $G$ satisfying $\mu_\caO(\caO)=1$.}
The map $\phi$ is continuous when $\mathcal{H}(G,\mathcal O)_\C$ is endowed with the topology induced by $\Vert \argu \Vert_1$, where 
$$\Vert f\Vert_1\ce\int_G |f(w)|\dd\mu_O(w), \quad \forall\, f\in \mathcal{H}(G,O)_\C.$$
The uniform closure $\overline{\mathcal{H}}(G,\mathcal O)$ of $\phi(\mathcal{H}(G,\mathcal O)_\C^{op})$ in $\mathcal{B}(G,\mathcal O)$, endowed with the standard operator norm, is a $C^*$-subalgebra of $\mathcal{B}(G,\mathcal O)$. Such an operator algebra is called the \emph{$C^*$-Hecke algebra associated to the Hecke pair $(G,\caO)$}.

Since the weak operator topology is coarser than the uniform operator topology, the map $\phi$ in \eqref{eq:phi} remains continuous if $\mathcal{B}(G,\caO)$ carries the weak operator topology. Denote by
$W(G,\caO)$ the $W^*$-algebra generated by $\phi(\mathcal{H}(G,\caO)_\C^{op})$ (cf.~$\S$\ref{ss:w*}) and call it the \emph{$W^*$-algebra 
associated to the Hecke pair $(G,\caO)$}.
\subsection{The trace map on $M_n(W(G,\caO))$} 
Let $G$ be a unimodular t.d.l.c.~group.
In {\cite[\S3.7]{ccw:euler}}, for every $\caO\in\CO(G)$ the authors defined a $\mathbb{C}$-valued trace map on the $C^*$-Hecke algebra $\overline{\mathcal{H}}(G,\caO)$.
 Here we extend such a map to the weak closure of $\mathcal{H}(G,\caO)_\C$.

\begin{prop}\label{prop: TrW} 
The $\C$-linear map
$$\tau_0\colon\phi(\mathcal{H}(G,\caO)_\C^{op})\to\C,\quad \tau_0(\phi(f))=f(1),$$
can be extended to the $\C$-linear map
$\tau\colon W(G,\caO)\longrightarrow \C$ given by 
$$\tau(F)=\langle F([1\dbr), [1\dbr\rangle\quad\text{for  } F\in W(G,\caO),$$
which is a positive and faithful trace  with $\tau(F^*)=\overline{\tau(F)}$ for all $F\in W(G,O)$.
\end{prop}
\begin{proof} Clearly, the map $\tau$ is $\C$-linear.
Moreover, let $\mathcal R$ denote a set of representatives of $\caO\backslash G/\caO$ and $f=\sum_{r\in\mathcal R}f(r)\dbl r\dbr\in\mathcal{H}(G,\caO)_\C$, where $f$ vanishes for all but finitely many $r$ (cf.~\eqref{eq:hgo}). Then, one computes $\tau(\phi(f))$ as
\begin{alignat}{2}\label{eq:tau hgo}
\langle [1\dbr\ast_{\mu_\caO}f, [1\dbr\rangle & =  \sum_{r\in\mathcal R}f(r)\langle [1\dbr\ast_{\mu_\caO}\dbl r\dbr,[1\dbr\rangle = \sum_{r\in\mathcal R}f(r)\langle \dbl r\dbr,[1\dbr\rangle\\
 & =  \sum_{r\in\mathcal R}f(r)\int_G I_{\caO r\caO}(w)I_ {\caO}(w)\dd\mu_\caO(w)\notag\\
 &=f(1)\int_G (I_ {\caO}(w))^2\dd\mu_\caO(w)
 =f(1)\mu_{\caO}(\caO)=f(1). \notag
\end{alignat}
Hence, $\tau(\phi(f))=\tau_0(\phi(f))$ for every $f\in \mathcal{H}(G,\caO)_\C$. In general, the multiplication map in $W(G,\caO)$ is not continuous w.r.t.~the weak topology as a function of two variables. However, it is continuous as a function of one variable when the other is held fixed. Then, for every fixed $F\in W(G,\caO)$ and every sequence $\{G_n\}_{n\geq 1}$ in $\phi(\mathcal{H}(G,\caO)_\C^{op})$ which is weakly convergent to $G\in W(G,\caO)$, the sequence $\tau(F G_n)-\tau(G_n F)$ converges to $\tau(F G)-\tau(G F)$. Hence, to prove that $\tau$ satisfies the trace property on $W(G,\caO)$, it suffices to verify that $\tau(F_1F_2)=\tau(F_2F_1)$ for arbitrary elements $F_1, F_2$ of a basis of $\phi(\mathcal{H}(G,\caO)_\C^{op})$. 
More precisely, it suffices to show that $\langle [1\dbr \ast_{\mu_\caO} \dbl g\dbr, [1\dbr\ast_{\mu_{\caO}}\dbl h\dbr^* \rangle=\langle [1\dbr\ast_{\mu_{\caO}}\dbl h\dbr, [1\dbr\ast_{\mu_{\caO}}\dbl g\dbr^*\rangle$ for all $g,h\in G$, and the end of the proof of \cite[Theorem~3.10]{ccw:euler} can be transferred here verbatim. 

For an arbitrary $F\in W(G,\caO)$, one has
\begin{equation}\label{eq:t faith}
    \tau(F^*F)=\langle\,F^*F([1\dbr), [1\dbr\,\rangle=\langle\,F([1\dbr), F([1\dbr)\,\rangle\geq 0
    \end{equation}
and $\tau$ is clearly positive. Moreover, if $\tau(F^*F)$ vanishes, then $F([1\dbr)=0$ (cf.~\eqref{eq:t faith}) and, since $F$ commutes with the left $G$-action, $F([g\dbr)=0$ for every $g\in G$. By linearity, $F$ vanishes over $C_c(G,\C)^\caO$ and then, by \cite[Fact 3.3]{ccw:euler}, $F=0$. Hence, $\tau$ is a faithful trace on $W(G,\caO)$.

Finally, for every $F\in W(G,\caO)$ one has
  $$\tau(F^*)=\langle\, F^*([1\dbr), [1\dbr\,\rangle=\langle\, [1\dbr,F([1\dbr)\,\rangle=\overline{\langle\, F([1\dbr), [1\dbr\,\rangle}=\overline{\tau(F)}.$$
\end{proof}
The following fact is analogous to~\cite[Fact~2.2]{ccw:euler}.
\begin{cor}\label{cor: TrW^n}
 The map
 $$\bar\tau\colon M_n(W(G,\caO))\to\C,\quad [F_{ij}]\mapsto \sum_{i=1}^n \tau(F_{ii}),$$
 is a positive and faithful trace over the  $W^*$-algebra $M_n(W(G,\caO))$ satisfying $\bar\tau(F^*)=\overline{\bar\tau(F)}$, for every $F\in M_n(W(G,\caO))$.
\end{cor}
\begin{lem}\label{lem: DiseqTr}
For every $F\in M_n(W(G,\caO))$, one has
$$|\bar\tau(F)|\leq n\cdot\nor F\nor,$$
where $ \nor F\nor=\max\{\nor F(v)\nor_2:\, v\in (L^2(G,\C)^\caO)^n, \nor v\nor_2=1\}$ and $\nor v\nor_2\ce\sqrt{\sum_{i=1}^n \nor v_i\nor_2^2}$, for every $v=[v_1,...,v_n]^t\in (L^2(G,\C)^\caO)^n$.
\end{lem}
\begin{proof}
By Corollary~\ref{cor: TrW^n} and the Cauchy--Schwarz inequality in $L^2(G,\C)^\caO$,
$$|\bar\tau(F)|\leq \sum_{j=1}^n |\tau(F_{jj})|=\sum_{j=1}^n|\langle \,F_{jj}([1\dbr), [1\dbr\,\rangle|\leq \sum_{j=1}^n\nor F_{jj}([1\dbr)\nor_2.$$
For $j=1,\ldots,n$, let $[1\dbr_j$ be the vector in the $n$-fold of $L^2(G,\C)^\caO$ having the $i^{th}$ component equals to $[1\dbr\cdot\delta_{ij}$ for every $i=1,\dots,n$. Since $\nor{[1\dbr_j}\nor_2=1$, one has
$$\nor F\nor\geq \nor F([1\dbr_j)\nor_2=\sqrt{\sum_{i=1}^n \nor F_{ij}([1\dbr)\nor_2^2}\geq \nor F_{jj}([1\dbr)\nor_2.$$
As a conclusion,
$\displaystyle |\bar\tau(F)|\leq\sum_{j=1}^n\nor F_{jj}([1\dbr)\nor_2\leq n\cdot \nor F\nor.$
\end{proof}
From now on a self-adjoint idempotent in a $W^\ast$-algebra is called a \emph{projection}.
\begin{thm}\label{thm: Kapl}
For every idempotent $e\in M_n(W(G,\caO))$, there exists a projection $p\in M_n(W(G,\caO))$ such that $e M_n(W(G,\caO))=p M_n(W(G,\caO))$
and $\bar\tau(e)=\bar\tau(p)$.
\end{thm}
\begin{proof} 
(cf. \cite[Theorem~26]{kap}). For $p=ee^*(1+(e-e^*)(e^*-e))^{-1}$ one has $p=ep$ and $e=pe$.
  As $\bar\tau$ is a trace, $\bar\tau(p)=\bar\tau(ep)=\bar\tau(pe)=\bar\tau(e).$
\end{proof}
\begin{cor}\label{cor: Kapl}
For every idempotent $e\in M_n(W(G,\caO))$,
$\bar\tau(e)\in [0,n]$. Moreover, $\bar\tau(e)=0$ if and only if $e=0$.
\end{cor}
\begin{proof}
By Theorem \ref{thm: Kapl}, for every idempotent $e\in M_n(W(G,\caO))$ there exists a projection $p\in M_n(W(G,\caO))$ such that $\bar\tau(e)=\bar\tau(p)=\bar\tau(p^*p)$. Then,
\begin{eqnarray*}
\bar\tau(e)  =& \displaystyle \sum_{1\leq j\leq n}\tau((p^*p)_{jj}) & =  \sum_{1\leq i,j\leq n}\tau((p^*)_{ji}p_{ij})\\
                      =&  \displaystyle \sum_{1\leq i,j\leq n}\tau(p_{ij}^*p_{ij})
                                 & =  \sum_{1\leq i,j\leq n}\langle\, p_{ij}([1\dbr), p_{ij}([1\dbr)\,\rangle.
\end{eqnarray*} 
 Hence $\bar\tau(e)\geq 0$ and $\tau(e)=0$  if, and only if, $p_{ij}=0$ for every $1\leq i,j\leq n$. As a consequence, $\bar\tau(e)=0$ if, and only if, $0=p M_n(W(G,\caO))=e M_n(W(G,\caO)),$
which is in turn equivalent to the condition that $e=0$.

To prove that $\bar\tau(e)\leq n$, it suffices to observe that $1-e$ is an idempotent of $M_n(W(G,\caO))$ and so
$0\leq\bar\tau(1-e)=n-\bar\tau(e).$
\end{proof}
\begin{rem}\label{rem: TrProj1} We collect some properties that will be useful later on.
\begin{itemize}
\item[(a)] From the proof of Corollary \ref{cor: Kapl}, it follows that
$$\bar\tau(p)=\sum_{1\leq i,j\leq n}\nor p_{ij}([1\dbr)\nor_2^2,$$
for every projection $p\in M_n(W(G,\caO))$.
\item[(b)] Let $p,q\in M_n(W(G,\caO))$ be projections. By \cite[Ch.~1, Proposition~1]{ber:baer}, one has
$$p M_n(W(G,\caO))\subseteq q M_n(W(G,\caO))\Longleftrightarrow p=qp.$$
Consequently, if $p M_n(W(G,\caO))\subseteq q M_n(W(G,\caO))$ then $\bar\tau(p)\leq \bar\tau(q)$. 
Indeed, since $p=qp$, one obtains that
$$pq=p^*q^*=(qp)^*=p^*=p=qp$$ 
and thus $q-p$ is a projection. By Corollary~\ref{cor: Kapl}, one concludes that
\begin{equation*}\label{eq: TrPrj}
\bar\tau(q)-\bar\tau(p)=\bar\tau(q-p)\geq 0.
\end{equation*}
\end{itemize}
\end{rem}
\section{Euler--Poincaré characteristic of unimodular t.d.l.c.~groups with $\ccd_\Q(G)\leq 1$}
\label{s:euchar}
\subsection{The Hattori--Stallings rank of discrete $\Q[G]$-modules of type $\FP$}
\label{ss:hs}
Let $G$ be a unimodular t.d.l.c.~group and $\mathcal{O}\in\CO(G)$. Denote by $\boh(G)$ the $\Q$-vector space $\Q\cdot\mu_{\mathcal O}$. Note that $\boh(G)$ consists of the rational multiples of the Haar measures on $G$ that take rational values on compact open subgroups, so $\boh(G)$ does not depend on the choice of $\mathcal{O}\in\CO(G)$.
For every $P\in \mathrm{ob}(\QGdis)$ finitely generated and projective, its \textit{Hattori--Stallings rank} $\tilde{\rho}(P)$ is defined as
$$\tilde\rho(P)\ce\rho_P(\iid_P)\in\boh(G),$$
where $\rho_P$ is a suitable map from $\mathrm{End}_G(P)$ to $\boh(G)$ defined in \cite[\S4.3]{ccw:euler}. For instance,  $\Q[G/\mathcal{O}]$ has Hattori--Stallings rank equal to $\mu_{\mathcal O}\in\boh(G)$ for every compact open subgroup $\caO$ of $G$. 
\begin{fact}[\protect{See \cite[Proof of Theorem~4.4]{ccw:euler}}]
\label{fact: EqTraces}
Let $G$ be a unimodular t.d.l.c.~group, and 
let $P\in\mathrm{ob}(\QGdis)$ be finitely generated and projective. Then there exist $n\in\Z_{\geq 1}$, $\caO\in\CO(G)$ and $e\in M_n(\mathcal{H}(G,\caO)_\Q)$ idempotent such that 
$P\simeq e\cdot\Q[G/\caO]^n$ and therefore
\begin{equation}\label{eq: rhoTilde}
\tilde{\rho}(P)=\bar{\tau}(e)\cdot\mu_\caO.
\end{equation}
\end{fact}
It is possible to extend the definition of the Hattori--Stallings rank
to every $M\in\mbox{ob}(\QGdis)$ of type $\FP$ as follows. Let
$P_0,...,P_n\in \mathrm{ob}(\QGdis)$ be finitely generated projective  such that 
$$
\xymatrix{0\ar[r]&P_n\ar[r]&\ldots\ar[r]&P_1\ar[r]&P_0\ar[r]&M\ar[r]&0}
$$
is a projective resolution of $M$ in $\QGdis$ (cf.~\S\ref{s:Rat}). Then,
one may define
\begin{equation}\label{eq: HSFP}
    \tilde{\rho}(M)\ce\sum_{i=0}^{n}(-1)^i\tilde{\rho}(P_i).
\end{equation}
By the additivity of $\tilde{\rho}$ (cf.~\cite[Equation~(5.1)]{ccw:euler}) and \cite[Ch.VIII, Lemma~4.4]{bro:coho}, the definition in \eqref{eq: HSFP} is independent of the choice of the projective resolution of $M$.
 The value $\tilde{\rho}(M)$ is sometimes called the \emph{Lefschetz number} of the module $M$.
From the Horseshoe Lemma one concludes the following property.
\begin{prop}\label{prop: addHS}
Let $0 \to A\to  B \to C\to 0$ be a short exact sequence in $\QGdis$, and assume that $A$ and $C$ are of type $\FP$. Then, so is B and
$\tilde{\rho}(B)=\tilde{\rho}(A)+\tilde{\rho}(C).$
\end{prop}
In the case that $G$ is a unimodular t.d.l.c.~group of type $\FP$, the trivial module $\Q$ admits a finite projective resolution $P_\bullet\rightarrow\Q$ in $\QGdis$ and the Euler--Poincaré characteristic of $G$ can be defined as $$\tilde{\chi}_G=\sum_{i\geq 0}(-1)^i\tilde{\rho}(P_i),$$ that is the Lefschetz number $\tilde\rho(\Q)$ of the trivial discrete left $\QG$-module $\Q$ (cf.~\eqref{eq: HSFP}). 
 For instance, if $G$ is profinite then $\tilde{\chi}_G=1\cdot\mu_G$, where $\mu_G$ denotes the Haar measure of $G$ such that $\mu_G(G)=1$ (cf.~\cite[Remark~1.3]{ccw:euler}).

 \smallskip
{From now on, given an open subgroup $H$ of $G$, the standard Haar measure on $H$ will be the restriction of the given Haar measure $\mu$ on $G$.}

\begin{prop}[\protect{\cite[Proposition~2.2.2]{cast:phd}}]\label{prop: chiAH}
Let $G$ be a unimodular t.d.l.c.~group of type $\FP$. 
\begin{itemize}
    \item[(a)] If $G=X\ast_U Y$ for some open subgroups $X,Y,U$ of $G$ of type $\FP$, then 
   $$\tilde{\chi}_G=\tilde{\chi}_X+\tilde{\chi}_Y-\tilde{\chi}_U;$$
    
    \item[(b)] If $G=X\ast_U^t$ for some open subgroups $X,U$ of $G$ of type $\FP$, then 
    $$\tilde{\chi}_G=\tilde{\chi}_X-\tilde{\chi}_U.$$
\end{itemize}
{The equalities in (a)~and~(b) must be interpreted as equalities among the rational coefficients 
 of the given Euler--Poincar\'e characteristics, all computed with respect to the same compact open subgroup of~$U$.}
\end{prop}
\begin{proof} 
{In case~(a), as well as in case~(b), $G$ acts without inversions on the associated Bass--Serre tree. Therefore, the proof strategye similar and we only provide a proof in case~(b). Since $G=X\ast_U^t$, one has the following short exact sequence in $\QGdis$: 
$$0\to\Q[G/U]\to\Q[G/X]\to\Q\to 0.$$
By~\cite[Corollary~3.18]{ccc:finite}, the three modules before are of type~$\FP$. Therefore, by Proposition~\ref{prop: addHS}, $$\tilde\chi_G=\tilde\rho(\Q)=\tilde\rho(\Q[G/X])-\tilde\rho(\Q[G/U]).$$
Let $H$ be an arbitrary open subgroup of $G$ of type $\FP$ and consider on it the Haar measure~$\mu^H$ which arises as restriction of the Haar measure~$\mu$ on~$G$.
We claim that, for some (and hence all) $\caO\in\CO(H)$, there is  $q(\caO)\in\Q$ such that $$\tilde\chi_H=q(\caO)\cdot \mu^H_\caO\quad\text{and}\quad\tilde\rho(\Q[G/H])=q(\caO)\cdot\mu_\caO.$$
To compute $\tilde\rho(\Q[G/H])$ we only need a finite projective resolution of $\Q[G/H]$ in $\QGdis$. To build such a resolution we start by considering a  finite projective resolution $P_\bullet\to\Q$ of discrete $\Q[H]$-modules, which does exist since $H$ is of type $\FP$, and then we apply the induction functor $\QG\otimes_H\argu$ to $P_\bullet\to\Q$. Note that this works as $\QG\otimes_H\argu$ is exact and left adjoint to the restriction functor from $\QGdis$ to ${}_{\Q[H]}\mathrm{\textbf{dis}}$ and so it preserves projective objects (cf.~\cite[\S~2.4]{cw:qrat}). Therefore, it suffices to prove that the Hattori--Stallings rank behaves well under induction, i.e., for some (and hence all) $\caO\in\CO(H)$ there exists $q_P(\caO)\in\Q$ such that $\tilde\rho(P)=q_P(\caO)\cdot \mu^H_\caO$ and $\tilde\rho(\Q[G]\otimes_H P)=q_P(\caO)\cdot \mu_\caO$. 
It follows by means of an analogous argument to the one used in the proof of the claim in the abstract case (cf.~{\cite[Ch.~IX, Prop.~2.3]{bro:coho}}). In the t.d.l.c.~case the free modules must be replaced by rational discrete permutation modules with compact open stabilisers (cf.~\cite[\S 3.2]{cw:qrat}).}
\end{proof}
\subsection{Discrete left augmentation $\QG$-modules}
Let $G$ be a t.d.l.c.~group. For every $U\in\CO(G)$, the map
\begin{equation}
\eps_U\colon C_c(G,\Q)^U\to\Q,\quad f\mapsto\int_G f(\omega)\dd\mu_U(\omega)
\end{equation}
is a surjective morphism in $\QGdis$ mapping every  $[g\dbr$ to $1$. The kernel of $\eps_U$, denoted by
$N_U^G$, is called the {\it discrete left augmentation  $\QG$-module relative to $U$.} Note that $C_c(G,\Q)^U$ is isomorphic to $\Q[G/U]$ in $\QGdis$.
\begin{prop}\label{prop:Ngu}
Let $G$ be a t.d.l.c.~group and $U$ a compact open subgroup. 
\begin{enumerate}
    \item[(a)] If $N^G_U$ is finitely generated, there are $n\in\Z_{\geq 0}$ and $\{s_1,\ldots,s_n\}\subseteq G\setminus U$ such that the set $\{I_{s_iU}-I_U\mid 1\leq i\leq n\}$ generates $N^G_U$ as discrete left $\QG$-module. In particular, one has an epimorphism of discrete $\QG$-modules
    \begin{equation}\label{eq:pOU}
        \begin{array}{rccl}
           p_{\caO,U}\colon & (C_c(G,\Q)^\caO)^n & \longrightarrow & N^G_U\\
             & [0,\ldots, \stackrel{j\text{-th}}{I_{\caO}}, \ldots, 0]^t & \longmapsto & I_U-I_{s_jU}
        \end{array}
    \end{equation}
    from the $n$-fold of $C_c(G,\Q)^\caO$ to $N^G_U$, where the subgroup $\caO$ is equal to $U\cap s_1Us_1^{-1}\cap\ldots \cap s_nUs_n^{-1}$.
    \item[(b)] $G$ is compactly generated if, and only if, $N^G_U$ is finitely generated.
    \item[(c)] $N^G_U$ is projective if, and only if, $\ccd_\Q(G)\leq 1$.
\end{enumerate}
\end{prop}
\begin{proof}
    For (a) and (b) see \cite[Proposition~5.3]{cw:qrat}, for (c) see \cite[Lemma~3.6]{cw:qrat}.
\end{proof}
Consequently, if $G$ is compactly generated and $\caO\leq U$ are compact open subgroups, one defines the endomorphism  
\begin{equation}\label{eq:pi}
\pi\colon\xymatrix{(C_c(G,\Q)^\caO)^n\ar[r]^-{p_{\caO,U}}&N^G_U\leq C_c(G,\Q)^U\ar[r]^-{\eta_{U,\caO}}&C_c(G,\Q)^\caO\ar[r]^-{\xi}& (C_c(G,\Q)^\caO)^n},
\end{equation}
where $\eta_{U,\caO}\colon C_c(G,\Q)^U\to C_c(G,\Q)^\caO$
is the injective homomorphism given by 
$$\eta_{U,\caO}(I_{gU})=\frac{1}{|U:\caO|}\sum_{r\in U/\caO}I_{gr\caO}$$
(as defined in \cite[\S4.2]{cw:qrat}) and $\xi([g\dbr)=[[g\dbr,0,\ldots,0]^t$ embeds into the first component.
\begin{fact}\label{fact:Ngu}
{Let $G, \caO$ and $U$ as before.}
According to Proposition \ref{prop:Ngu}, the following properties hold:
\begin{itemize}
    \item[(i)] the image of $\pi$ is isomorphic to $N^G_U$ in $\QGdis$;
    \item[(ii)] if $\ccd_\Q(G)\leq1$, there exists a homomorphism in $\QGdis$
    $$\iota\colon\image(\pi)\to (C_c(G,\Q)^\caO)^n$$
    such that $\pi\iota=\mathrm{id}_{\image(\pi)}$;
    \item[(iii)] the element $\alpha\ce\iota\pi$
    satisfies $\alpha^2=\alpha$ and $\pi\alpha=\pi$.
\end{itemize}
\end{fact}
\begin{proof}
The only property which does not follow directly by construction is the existence of the morphism $\iota.$ It is instead a consequence of the fact that
 $\image(\pi)$ is projective in $\QGdis$ (cf.~Proposition~\ref{prop:Ngu}(c)).
\end{proof}
\begin{rem}\label{rem: PrelApprox}
By \S\ref{ss:hs}~and~Corollary~\ref{cor: Kapl}, 
one has $\tilde{\rho}(N_U^G)=\bar\tau(\alpha)\cdot\mu_\caO\geq 0$. 
\end{rem}
From now on, the maps $\pi$ and $\alpha$ will be identified with their images in the operator algebra $M_n(\phi(\mathcal{H}(G,\caO)_\Q^{op}))\subseteq M_n(W(G,\caO))$ (cf.~Fact~\ref{fact: EndoH(G,O)nxn}(b) and \eqref{eq:phi}). In particular, $\pi=[\pi_{ij}]$ is the matrix with entries in $\phi(\mathcal{H}(G,\caO)_\Q^{op})$ defined by
\begin{equation}\label{eq:piij}
    \pi_{ij}([g\dbr)\ce\begin{cases}
    \hfill 0\hfill &\text{if $i>1,$}\\ \displaystyle\frac{1}{|U:\caO|}\sum_{r\in U/\caO}([gr\dbr-[gs_jr\dbr)&\text{if $i=1$.}\\
    \end{cases}
\end{equation}
for every $g\in G$. Recall that $\{[g\dbr=I_{g\caO}\mid g\in G/\caO\}$ is an orthonormal basis of the Hilbert space $L^2(G,\C)^\caO$ (cf.~\eqref{eq:L2O}).
For simplicity, let $\pi_j\ce\pi_{1j}$ for every $1\leq j\leq n$.
The following fact is relevant for the proof of Proposition \ref{prop: Approx}. 
In particular, we introduce in item~(c) the matrix $e\in M_n(W(G,\caO))$ such that $M_n(W(G,\caO))e\pi=M_n(W(G,\caO))\pi$ and such that $\bar{\tau}(e\pi)$ produces a lower bound for the rank $\tilde{\rho}(N_U^G)$ which is independent of the number $n$ of generators of the discrete left $\QG$-module $N_U^G$ (cf.~\eqref{eq:approx}).
\begin{fact}\label{fact:Ngu analitic}
Using the notation from~\eqref{eq:pi}, the following properties hold:
\begin{itemize}
    \item[(a)] $\pi^*M_n(W(G,\caO))\subseteq\alpha^*M_n(W(G,\caO))$.
    \item[(b)] There exists a projection $p\in M_n(W(G,\caO))$ such that
    $$\alpha^*M_n(W(G,\caO))=pM_n(W(G,\caO))\quad\text{and}\quad \bar\tau(p)=\bar\tau(\alpha^*)=\bar\tau(\alpha).$$
    \item[(c)] If $e$ denotes the matrix such that $e_{ij}=1_{W(G,\caO)}$ for all $1\leq i,j\leq n$, then one has $n=|U:\caO|\cdot\bar\tau(e\pi)$.
\end{itemize}
\end{fact}
\begin{proof}
By Fact~\ref{fact:Ngu}(iii),  $M_n(W(G,\caO))\pi=M_n(W(G,\caO))\pi\alpha\subseteq M_n(W(G,\caO))\alpha$ and (a) holds. Again, Fact~\ref{fact:Ngu}(iii) implies that $\alpha^*$ is an idempotent. Then Theorem \ref{thm: Kapl} applies and, since $\bar{\tau}(\alpha)\in \Q$, (b) is proved. Finally, (c) follows from  $\bar\tau(e\pi)=\sum_{i=1}^n\bar\tau(\pi_i)$ and $\bar{\tau}(\pi_i)=\langle\pi_i([1\dbr),[1\dbr\rangle=1/|U:\caO|$. 
\end{proof}
\begin{prop}\label{prop: Approx}
Let $G$ be a compactly generated unimodular t.d.l.c.~group with $\ccd_\Q(G)\leq 1$. For every compact open subgroup $U\neq G$, one has that
\begin{equation}\label{eq:approx}
\tilde{\rho}(N_U^G)\geq \frac{1}{2}\mu_U.
\end{equation}
\end{prop}

\begin{proof}
We use here the notation of Fact~\ref{fact:Ngu}, Remark~\ref{rem: PrelApprox} and Fact~\ref{fact:Ngu analitic}. In particular, $\alpha\in M_n(W(G,\caO))$ is the idempotent such that $\tilde\rho(N^G_U)=\bar\tau(\alpha)\mu_{\caO}$ and $p\in M_n(W(G,\caO))$ is the projection such that
    $\alpha^*M_n(W(G,\caO))=pM_n(W(G,\caO))$ and $\bar\tau(p)=\bar\tau(\alpha^*)=\bar\tau(\alpha)$. Suppose that $N^G_U$ is $\QG$-generated by the elements $I_{s_iU}-I_U\in C_c(G,\Q)^U$, for $1\leq i\leq n$ (cf. Proposition \ref{prop:Ngu}).

Denote by $e\in M_n(W(G,\caO))$ the matrix such that $e_{ij}= 1_{W(G,\caO)}$ for all $1\leq i,j\leq n$, and let $\epsilon>0$. By \cite[Corollary,~p.~43]{ber:baer}, there exists a non-zero projection $f=f_\epsilon\in \pi^*e^*M_n(W(G,\caO))\subseteq \pi^*M_n(W(G,\caO))\subseteq pM_n(W(G,\caO))$ such that $\nor e\pi-e\pi f\nor<\epsilon$. Hence, by Fact~\ref{fact:Ngu analitic} and since $|\bar\tau(e\pi-e\pi f)|\leq n\cdot\nor e\pi-e\pi f\nor<n\cdot\epsilon$ (cf.~Lemma~\ref{lem: DiseqTr}), one deduces that
\begin{equation}\label{eq:meq1}
    \frac{n}{|U:\caO
|}=\bar\tau(e\pi)=|\bar\tau(e\pi)|\leq |\bar\tau(e\pi-e\pi f)|+|\bar\tau(e\pi f)|\leq n\cdot\epsilon+|\bar\tau(e\pi f)|.
\end{equation}
Notice that all the rows in the matrix $e\pi f$ are equal to the vector 
$$\Bigg[\sum_{1\leq i\leq n}\pi_if_{i1},\ldots,\sum_{1\leq i\leq n}\pi_if_{in}\Bigg].$$ 
Hence the properties of $\bar\tau$ imply that
\begin{equation*}
|\bar\tau(e\pi f)|   
= \Bigg |\sum_{1\leq i,j\leq n}\tau(\pi_if_{ij})\Bigg |
 = \Bigg |\sum_{1\leq i,j\leq n}\tau(f_{ij}\pi_i)\Bigg|=\Bigg| \sum_{1\leq i,j\leq n} \langle f_{ij}\pi_i([1\dbr), [1\dbr\rangle\Bigg|
 \end{equation*}
 and so the Cauchy--Schwarz inequality yields
 \begin{equation}\label{eq:meq2}
|\bar\tau(e\pi f)|\leq \sum_{1\leq i,j\leq n} |\langle \pi_i([1\dbr), f_{ij}^*([1\dbr)\rangle|
                \leq  \sum_{1\leq i,j\leq n} \nor\pi_i([1\dbr)\nor_2\cdot \nor f_{ij}^*([1\dbr)\nor_2.
\end{equation}
By (\ref{eq:piij}), for each $1\leq i\leq n$ one computes
\begin{equation*}
\nor\pi_i([1\dbr)\nor_2^2=\frac{1}{|U:\caO|^2}\Bigg\langle \sum_{r\in U/\caO} ([r\dbr-[s_ir\dbr), \sum_{r'\in U/\caO} ([r'\dbr-[s_ir'\dbr)\Bigg\rangle=\frac{2}{|U:\caO|},
\end{equation*}
having $[r\dbr\neq[s_ir'\dbr$ and $[r'\dbr\neq[s_ir\dbr$ for all $r,r'\in U/\caO$ (because $U\cap s_iU=\emptyset$).

Then, the inequalities \eqref{eq:meq1} and \eqref{eq:meq2} imply that
\begin{equation}\label{eq:meq3}
\frac{n}{|U:\caO|}
 \leq  n\cdot \epsilon + \sqrt{\frac{2}{|U:\caO|}}\Bigg( \sum_{1\leq i,j\leq n}\nor f_{ij}^*([1\dbr)\nor_2\Bigg).
 \end{equation}
Using the explicit equivalence between the $\ell^1$-norm and the $\ell^2$-norm in $\R^{n^2}$ for the vector $[\nor f_{ij}^*([1\dbr)\nor_2]_{1\leq i,j\leq n}$, one obtains
$$\sum_{1\leq i,j\leq n}\nor f_{ij}^*([1\dbr)\nor_2\leq n\cdot \sqrt{\sum_{1\leq i,j\leq n}\nor f_{ij}^*([1\dbr)\nor_2^2}.$$
Therefore, \eqref{eq:meq3}~and~Remark~\ref{rem: TrProj1}(a) yield
\begin{equation}\label{eq:meq4}
              \frac{n}{|U:\caO|}   \leq  n\cdot\Bigg (\epsilon + \sqrt{\frac{2}{|U:\caO|}\sum_{1\leq i,j\leq n}\nor f_{ij}^*([1\dbr)\nor_2^2}\Bigg )\\
                 =  n\cdot\Bigg (\epsilon + \sqrt{\frac{2}{|U:\caO|}\bar\tau(f)}\Bigg ),
\end{equation}
for every $\epsilon>0$. 

 Since $\bar\tau(\alpha)=\bar\tau(p)\geq \bar\tau(f)$ for every $\epsilon>0$ (cf.~Fact~\ref{fact:Ngu analitic}(b)~and~Remark~\ref{rem: TrProj1}(b)), it follows from \eqref{eq:meq4} that $$\bar\tau(\alpha)\geq \frac{1}{2|U:\caO|}.$$
By Remark \ref{rem: PrelApprox}, one concludes that  
$$\tilde{\rho}(N_U^G)=\bar\tau(\alpha)\mu_\caO\geq \frac{1}{2|U:\caO|}\mu_\caO=\frac{1}{2}\mu_U.$$
\end{proof}
\begin{cor}\label{cor: 1stEstimateChar}
Let $G$ be a compactly generated unimodular t.d.l.c.~group with $\ccd_\Q(G)\leq 1$. Then
$\displaystyle\tilde{\chi}_G\leq \frac{1}{2}\mu_U$
for every $U\in \CO(G)$, $U\neq G$. In particular,
\begin{equation}\label{eq:chi>0}
  \tilde{\chi}_G=\displaystyle\frac{1}{2}\mu_U \Longleftrightarrow\tilde{\rho}(N^G_U)=\displaystyle\frac{1}{2}\mu_U.
\end{equation}
\end{cor}
\begin{proof}
By Proposition~\ref{prop: addHS}~and~Proposition~\ref{prop:Ngu}, for every $U\in\mathcal{CO}(G)$ one has 
$$\tilde{\chi}_G=\tilde{\rho}(\Q)=\tilde{\rho}(C_c(G,\Q)^U)-\tilde{\rho}(N_U^G)=\mu_U-\tilde{\rho}(N_U^G).$$
The claim follows directly from Proposition~\ref{prop: Approx}.
\end{proof}
\begin{thm}\label{thm:chi}
Let $G$ be a compactly generated unimodular t.d.l.c.~group with $\ccd_\Q(G)=1$. Then 
$\tilde{\chi}_G\leq 0$. Moreover $\tilde{\chi}_G=0$ if, and only if, $G=X\ast_UY$ for some compact open subgroups $X,Y,U\leq G$ satisfying $|X:U|=|Y:U|=2$.
\end{thm}
\begin{proof}
First note that $\tilde{\chi}_\caO=\tilde{\rho}(\Q[G/\caO])=\mu_\caO$ for every compact open subgroup $\caO\leq G$ (cf.~\S\ref{ss:hs}).
By \cite[Theorem~$\textup A{}^*$]{IC:cone}, $G$ splits non-trivially over a compact open subgroup $U\in\mathcal{CO}(G)$.
Suppose first that $G=X\ast_U Y$, for some compact open subgroup $U$ and compactly generated open subgroups $X,Y$ of $G$ different to $U$ (cf.~\cite[Proposition~4.1]{IC:cone}). By Proposition~\ref{prop: chiAH}(a) and Corollary~\ref{cor: 1stEstimateChar}, one has 
\begin{equation}\label{eq:chi1}
   \tilde{\chi}_G=\tilde{\chi}_X+\tilde{\chi}_Y-\tilde\chi_U\leq \frac{1}{2}\mu_U+\frac{1}{2}\mu_U-\mu_U\leq 0. 
\end{equation}
Similarly, let $G=X\ast_U^t$ for some compact open subgroup $U\leq G$ and compactly generated open subgroup $X\leq G$ different to $U$. From Proposition~\ref{prop: chiAH}(b) and Corollary~\ref{cor: 1stEstimateChar} one deduces that
\begin{equation}\label{eq:chi2}
    \tilde{\chi}_G= \tilde\chi_X-\tilde\chi_U\leq \frac{1}{2}\mu_U-\mu_U<0.
\end{equation}
If $\tilde{\chi}_G=0$, then \eqref{eq:chi1} and \eqref{eq:chi2} imply that $G=X\ast_U Y$ for some compact open subgroup $U\leq G$ and compactly generated open subgroups $X,Y\leq G$ different to $U$ (cf.~\cite[Proposition~4.1]{IC:cone}).
Moreover, \eqref{eq:chi1} yields to
\begin{equation}\label{eq:chi3}
    z\ce\tilde{\chi}_X-\frac{1}{2}\mu_U=\frac{1}{2}\mu_U-\tilde{\chi}_Y.
\end{equation}
By Corollary~\ref{cor: 1stEstimateChar} applied to $X$ and $Y$, respectively, one deduces that $z$ is simultaneously non-positive and non-negative and therefore equal to $0$. Since $\ccd_\Q(X),\ccd_\Q(Y)\leq \ccd_\Q(G)=1$ (cf.~\cite[Proposition~3.7(c)]{cw:qrat}), the first part of the theorem implies that $\ccd_\Q(X)=\ccd_\Q(Y)=0$ and then $X$ and $Y$ are compact (cf.~\cite[Proposition~3.7(a)]{cw:qrat}). Thus,
$$\mu_{X}=\tilde{\chi}_X=\frac{1}{2}\mu_U=\tilde{\chi}_Y=\mu_X.$$
Since $\mu_{\caO_1}=|\caO_1:\caO_2|^{-1}\mu_{\caO_2}$ for all compact open subgroups $\caO_1,\caO_2\leq G$ with $\caO_2\leq \caO_1$, this yields the "only if" part of the claim. The "if" part is a direct consequence of Proposition~\ref{prop: chiAH}(a).
\end{proof}
\begin{rem}\label{rem:chi1/2}
    By Theorem~\ref{thm:chi}, the equivalence in \eqref{eq:chi>0} refines, for every $U\in \CO(G)$ with $U\neq G$, as follows:
    \begin{equation*}
  \tilde{\chi}_G=\displaystyle\frac{1}{2}\mu_U \Longleftrightarrow\tilde{\rho}(N^G_U)=\displaystyle\frac{1}{2}\mu_U \Longleftrightarrow |G:U|=2.
\end{equation*}
\end{rem}
\section{Accessibility: general results}
\label{s:acc}
\subsection{Graphs, trees and graphs of groups}\label{ss:gog}
The notion of graph here adopted is the one of J-P.~Serre (cf.~\cite[p.~13]{ser:trees}). 
Namely, a \emph{graph} $\Gamma=(\euV(\Gamma), \euE(\Gamma), o, t, \overline{\cdot})$ consists of a (non-empty) set of vertices $\euV(\Gamma)$, a set of edges $\euE(\Gamma)$, origin and terminus maps $o,t\colon\euE(\Gamma)\longrightarrow \euV(\Gamma)$ and an edge-reversing map $\overline{\cdot}\colon\euE(\Gamma)\longrightarrow \euE(\Gamma)$, which is an involution satisfying $o(\overline{\eue})=t(\eue)$ and $\overline{\eue}\neq \eue$ for every $\eue\in \euE(\Gamma)$.

A \emph{path} $\mathfrak{p}=(\eue_1,\ldots,\eue_n)$ is a finite sequence of edges such that $o(\eue_{i+1})=t(\eue_i)$, for every $1\leq i\leq n-1$. In particular, $\mathfrak{p}$ is \emph{reduced} if $\eue_{i+1}\neq\overline{\eue}_{i}$ for every $i<n$. A \emph{cycle} $(\eue_1,...,\eue_n)$ is a closed path, i.e., $o(\eue_1)=t(\eue_n)$.
A graph is \emph{connected} if it is either a $1$-point graph or, for every two vertices $v,w$, there exists a path $(\eue_1,...,\eue_n)$ with $o(\eue_1)=v$ and $t(\eue_n)=w$.
A \emph{tree} is a connected graph with no reduced cycles.

\smallskip

Adapting \cite[Definition~8, p.~37]{ser:trees} to the t.d.l.c.~context, one obtains the notion of \emph{graph of t.d.l.c.~groups} $(\gog, \Lambda)$. Namely, 
\begin{itemize}
  \item[(G1)] $\Lambda$ is a connected graph; 
  \item[(G2)] $\gog$ is a family of t.d.l.c.~groups $\gog_v$ ($v\in \euV(\Lambda)$) and $\gog_\eue$ ($\eue\in \euE(\Lambda)$) satisfying $$\gog_\eue=\gog_{\bar{\eue}},\quad\ \forall\ \eue\in \euE(\Lambda);$$
  \item[(G3)] for every $\eue\in \euE(\Lambda)$, there is a group monomorphism $\alpha_\eue:\gog_\eue\rightarrow \gog_{t(\eue)}$ that is continuous and open.
\end{itemize}
The graph of groups $(\gog, \Lambda)$ is called {\em proper} if, whenever $\alpha_\eue$ is surjective, then $o(\eue)=t(\eue)$.
For every graph of t.d.l.c.~groups $(\gog, \Lambda)$,
the (abstract) {\em fundamental group $\pi_1(\gog, \Lambda)$ of $(\gog, \Lambda)$} can be equipped with a group topology~$\tau$ such that
\begin{itemize}
    \item[(F1)] $\pi_1(\gog, \Lambda)$ is a t.d.l.c.~group; and 
    \item[(F2)] the natural inclusions $\gog_v\hookrightarrow\pi_1(\gog, \Lambda)$ and $\gog_\eue\hookrightarrow\pi_1(\gog, \Lambda)$ are continuous and open;
\end{itemize} 
(cf.~\cite[\S III.2, Proposition~1]{bou:top}). Moreover, the group topology~$\tau$ is uniquely determined by properties (F1) and~(F2).
From now on, the group $\pi_1(\gog, \Lambda)$ is regarded as a t.d.l.c.~group. 
The construction of the fundamental group depends on the choice of a maximal subtree $\euT\subseteq \Lambda$. Nevertheless, we shall denote the fundamental group by $\pi_1(\gog, \Lambda,\euT)$ only if the maximal subtree is relevant.

Several concepts concerning graphs of groups transfer verbatim to the t.d.l.c.~context. 
In particular, we recall the following construction (cf.~\cite[pp.~30-31]{Cohen}).
Let $(\gog, \Lambda)$ be a graph of t.d.l.c.~groups with compact edge-groups. Let $\mathcal{V}\subseteq \euV(\Lambda)$ and, for every $v\in\mathcal{V}$, assume that there is a graph of t.d.l.c.~groups $(\euH^{(v)}, \Gamma^{(v)})$ such that $\gog_v=\pi_1(\euH^{(v)}, \Gamma^{(v)})$. For every $\eue\in \euE(\Lambda)$ with $t(\eue)\in \mathcal{V}$, let $x=x_\eue\in \euV(\Gamma^{(t(\eue))})$ and $g_\eue\in \gog_{t(\eue)}$ be such that the image of $\gog_\eue$ in $\gog_{t(\eue)}$ lies in $g_{\eue}(\euH^{(t(\eue))})_xg_{\eue}^{-1}$. For every such $\eue$, since $\gog_\eue$ is profinite, the pair $(g_\eue, x_{\eue})$ exists by the Bruhat--Tits fixed point theorem (although it is not uniquely determined by~$\eue$).
  One defines the \emph{graph of t.d.l.c.~groups~$(\euH,\Gamma)$ obtained from~$(\gog,\Lambda)$ by expanding~$\mathcal{V}$ with~$\{\Gamma^{(v)}\}_{v\in\mathcal{V}}$ and with respect to $\{(g_{\eue}, x_{\eue})\}_{\eue\in\euE(\Lambda),\,t(\eue)\in\mathcal{V}}$} as follows. The graph $\Gamma$ has vertex-set $\bigsqcup_{v\in \mathcal{V}}\euV(\Gamma^{(v)})\sqcup \big(\euV(\Lambda)\setminus \mathcal{V}\big)$ and edge-set $\bigsqcup_{v\in\mathcal{V}}\euE(\Gamma^{(v)})\sqcup \euE(\Lambda)$. 
  The inversion map in~$\Gamma$ is the obvious one.
  The origin and terminus maps, denoted by~$o_\Gamma$ and $t_\Gamma$ respectively, are defined so that: 
  \begin{itemize}
      \item[(i)] for every $v\in\mathcal{V}$, $\Gamma^{(v)}$ is a subgraph of~$\Gamma$;
      \item[(ii)] the map~$t_\Gamma$ coincides with the terminus map~$t_\Lambda$ in~$\Lambda$ on every $\eue\in\euE(\Lambda)$ with~$t_\Lambda(\eue)\not\in\mathcal{V}$; and
      \item[(iii)] for every $\eue\in \euE(\Lambda)$ with $t_\Lambda(\eue)\in \mathcal{V}$, one has $t_\Gamma(\eue)=x_\eue$.
  \end{itemize}
The relevant graph of t.d.l.c.~groups~$(\euH,\Gamma)$ is defined as follows:
\begin{itemize}
    \item[(i)] for every $x\in (\euV(\Lambda)\setminus \mathcal{V})\sqcup \euE(\Lambda)$ let $\euH_x:=\gog_x$; for all $v\in \mathcal{V}$ and $x\in \euV(\Gamma^{(v)})\sqcup \euE(\Gamma^{(v)})$, let $\euH_x:=\euH_x^{(v)}$; 
    \item[(ii)] for every $\eue\in \euE(\Gamma)$, the monomorphism $\euH_\eue\hookrightarrow \euH_{t_\Gamma(\eue)}$ is the monomorphism $\gog_{\eue}\hookrightarrow \gog_{t_\Lambda(\eue)}$ (resp.~$\euH_{\eue}^{(v)}\hookrightarrow {\euH_{t_{\Gamma^{(v)}}(\eue)}^{(v)}}$) if $\eue\in \euE(\Lambda)$ and $t_\Lambda(\eue)\not\in\mathcal{V}$ (resp.~$\eue\in \euE(\Gamma^{(v)})$) and, if $\eue\in\euE(\Lambda)$ and $t_\Lambda(\eue)\in\mathcal{V}$, it is the monomorphism $\gog_\eue\hookrightarrow \gog_{t_\Lambda(\eue)}$ followed by the right-conjugation by~$g_\eue$.
\end{itemize}

\begin{lem}[\protect{cf.~\cite[Lemma~2]{Cohen}}]\label{lem:exp}
For $(\gog,\Gamma)$ and $(\euH,\Gamma)$ as before, one has $\pi_1(\gog,\Lambda)\simeq \pi_1(\euH,\Gamma)$. 
\end{lem}
\begin{prop}[\textbf{Generalised Mayer--Vietoris sequence}]\label{prop: MVSeq}
Let $G$ denote the fundamental group of the graph of t.d.l.c.~groups  $(\gog, \Lambda)$. Let $\euE^+\subseteq \euE(\Lambda)$ be an orientation on the edges of $\Lambda$, i.e., $|\euE^+\cap \{\eue,\overline{\eue}\}|=1$ for every $\eue\in \euE(\Lambda)$. For each
$M\in\mathrm{ob}(\QGdis)$, one has a long exact sequence 
\[\xymatrix@C-=1pc{\cdots\ar[r]&\displaystyle\prod_{v\in V(\Lambda)}\dH^n(\gog_v,M)\ar[r]&\displaystyle\prod_{\eue\in\euE^+} \dH^n(\gog_\eue,M)\ar[r]& \dH^{n+1}(G,M)\ar[r] &\cdots}\]
In particular,
\begin{equation}\label{eq:cdQ<}
    \sup_{v\in \euV(\Lambda)}\ccd_\Q(\gog_v)\leq\ccd_\Q(G)\leq \sup \{\ccd_\Q(\gog_v),1+\ccd_\Q(\gog_\eue)\mid v\in \euV(\Lambda),\eue\in \euE^+\}.
\end{equation}
\end{prop}
\begin{proof}
The first part of the statement is \cite[Proposition~5.4(a)]{cw:qrat}. From this, one easily deduces the second inequality of \eqref{eq:cdQ<}.
Finally, the first inequality in \eqref{eq:cdQ<} is due to \cite[Proposition~3.7(c)]{cw:qrat}. 
\end{proof}
Proposition~\ref{prop: chiAH} implies the following generalisation of \cite[Corollary~C]{ccw:euler}.
\begin{prop}\label{prop: chiGrOfGr}
Let $G$ be a unimodular t.d.l.c.~group which is isomorphic to $\pi_1(\gog, \Lambda)$, for some finite graph $(\gog,\Lambda)$ of t.d.l.c.~groups of type $\FP$. 
Then $G$ is of type $\FP$ and, for a given orientation $\euE^+\subseteq \euE(\Lambda)$, one has
$$\tilde{\chi}_G=\sum_{v\in \euV(\Lambda)}\tilde{\chi}_{\gog_v}-\sum_{\eue\in \euE^+}\tilde{\chi}_{\gog_\eue}.$$
\end{prop}
The equality in~Proposition~\ref{prop: chiGrOfGr} must be interpreted as an equality among the coefficients 
 of the given Euler--Poincar\'e characteristics, all computed with respect to a compact open subgroup in $\bigcap_{\eue\in\euE^+}\gog_{\eue}$. Here, every~$\gog_v$ and every~$\gog_\eue$ are regarded as open subgroups of~$G$.

 The following proposition focuses on a situation arising in Proposition~\ref{prop:inacc}, and will play a key role in the proof of Corollary~\ref{cor:inaccuni}. Contrary to the previous results, this situation can happen only for non-discrete t.d.l.c.~groups. 

\subsection{Accessibility of t.d.l.c.~groups}\label{sus:ab}
A t.d.l.c.~group $G$ is said to {\em split over a compact open subgroup} if one of the following two conditions occurs:
\begin{itemize}
    \item[($\alpha$)] $G\simeq\pi_1(\gog,\Lambda)$, where $\Lambda$ is the connected graph with two distinct vertices $v,w$ and one geometric edge $\{\eue, \bar{\eue}\}$, and $\gog_\eue\in\CO(G)$. In other words, $G$ is isomorphic to the t.d.l.c.~free product  with amalgamation $\gog_v\ast_{\gog_e}\gog_w$ (cf.~\cite[Definition 8.B.8]{codh:met}).
    \item[($\beta$)] $G\simeq\pi_1(\gog,\Lambda)$, where $\Lambda$ is the connected graph with one vertex $v$ and one geometric edge $\{\eue, \bar{\eue}\}$, and $\gog_\eue\in\CO(G)$. In other words, $G$ is isomorphic to the t.d.l.c.~HNN-extension $\gog_v\ast_{\gog_{\eue},\alpha_e}^{\eue}$, where $\eue$ is regarded as the stable letter of the extension (cf.~\cite[Definition 8.B.8]{codh:met}).
\end{itemize}
In particular, $G$ splits {\em non-trivially} if $\gog_\eue$ is strictly contained in each vertex-group.
By construction, the vertex-groups of $(\gog, \Lambda)$ are open subgroups of $G$. 
Moreover, by \cite[Proposition~4.1]{IC:cone}, if $G$ is compactly generated and splits as above, then the vertex-groups of $(\gog,\Lambda)$ are compactly generated.

\medskip

The decomposition theorem due to H.~Abels, which is included in \cite[Theorem~A${}^*$]{IC:cone}, implies that every compactly generated t.d.l.c.~group with more than one end splits non-trivially over a compact open subgroup. 
In particular, the notion of accessibility carries over to the t.d.l.c.~context as follows (\cite[Definition~8]{km:rough} and \cite{IC:cone}). 

A compactly generated t.d.l.c.~group $G$ is said to be {\em accessible} if it is isomorphic to the Bass--Serre fundamental group of a finite graph of t.d.l.c.~groups $(\gog,\Lambda)$ with compact edge-groups and whose vertex-groups have at most one end. 

According to the cohomological characterisation of compactly generated t.d.l.c.~groups with more than one end (cf.~\S\ref{s:Rat}~or~\cite[Theorem~A${}^*$]{IC:cone}), the concept of accessibility can be reformulated in the following terms.

\begin{defn}\label{def:acc}
A compactly generated t.d.l.c.~group $G$ is {\em accessible} if
\begin{itemize}
    \item[$(A1)$] $G$ is isomorphic to the fundamental group of a finite graph of t.d.l.c. groups $(\gog, \Lambda)$  with compact edge-groups; and 
    \item[$(A2)$] $\dH^1(\gog_v,\Bi(\gog_v))=0$ for every $v\in\euV(\Lambda)$.
    \end{itemize}
\end{defn}

\begin{fact}\label{fact:one end}
Let $G$ be a compactly generated t.d.l.c.~group with $\ccd_\Q(G)=1$. Then $G$ splits non-trivially over a compact open subgroup. 
Moreover, for every finite graph of t.d.l.c.~groups $(\gog, \Lambda)$ with compact edge-groups satisfying $G\simeq \pi_1(\gog,\Lambda)$ and every $v\in \euV(\Lambda)$, either $\gog_v$ is compact (if $\ccd_\Q(\gog_v)=0$) or it is compactly generated and splits non-trivially over a compact open subgroup (if $\ccd_\Q(\gog_v)=1$).
In particular $G$ is accessible if, and only if, $G$ is isomorphic to the fundamental group of a finite proper graph of profinite groups.
\end{fact}
\begin{proof}
If $G$ is compactly generated with $\ccd_\Q(G)=1$, then it is evidently of type $\FP$ and so $\dH^1(G,\Bi(G))\neq 0$ by \cite[Proposition~4.7]{cw:qrat}. By \cite[Theorem~A${}^*$]{IC:cone}, $G$ splits non-trivially over a compact open subgroup. Let $(\gog, \Lambda)$ be a graph of t.d.l.c.~groups with compact edge-groups satisfying $G\simeq \pi_1(\gog,\Lambda)$. Recall that, for every $v\in \euV(\Lambda)$, the group $\gog_v$ can be regarded as an open subgroup of $G\simeq \pi_1(\gog, \Lambda)$. Thus, $\gog_v$ is compactly generated and $\ccd_\Q(\gog_v)\leq 1$ (cf.~\cite[Proposition~4.1]{IC:cone}~and~\cite[Proposition~3.7(c)]{cw:qrat}). Therefore, $\gog_v$ either is compact (if $\ccd_\Q(\gog_v)=0$) or splits non-trivially over a compact open subgroup (if $\ccd_\Q(\gog_v)=1$). Hence the first and the second statements of the claim hold.
The last part of the assertion follows from Definition~\ref{def:acc} and the fact that a t.d.l.c.~group $H$ of type $\FP$ satisfying $\ccd_\Q(H)\leq 1$ and $\dH^1(H,\Bi(H))=0$ is necessarily compact (cf.~\cite[Proposition~3.7(a), Proposition~4.7]{cw:qrat}).
\end{proof}

The following fact has been proved in \cite[Proposition~5.6]{cw:qrat}.
\begin{fact}\label{fsct:vice}
Let $G$ be the t.d.l.c.~fundamental group of a graph of profinite groups. Then $\ccd_\Q(G)\leq 1$.
\end{fact}

As a consequence of Lemma~\ref{lem:exp}, one obtains the following result.
\begin{prop}\label{prop: AccFact}
 Suppose that a t.d.l.c.~group $G$ is isomorphic to $\pi_1(\gog,\Lambda)$, where $(\gog,\Lambda)$ is a finite graph of t.d.l.c.~groups satisfying the following properties:
\begin{itemize}
    \item[(i)] every edge-group $\gog_\eue$ is compact;
    \item[(ii)] every vertex-group $\gog_v$ is accessible.
\end{itemize}
Then $G$ is accessible.
\end{prop}

\begin{prop}\label{prop:deg}
    Let $G$ be a compactly generated unimodular t.d.l.c.~group of type $\FP$. Let $\{(\gog_n,\Lambda_n)\}_{n\geq 1}$ be an infinite sequence of finite graphs of t.d.l.c.~groups with compact edge-groups, $G\simeq\pi_1(\gog_n,\Lambda_n)$ for every $n\geq 1$. Moreover, assume that for every $n\geq 2$ there are bijections $\varphi_n\colon \euV(\Lambda_n)\to \euV(\Lambda_1)$ and $\psi_n\colon \euE(\Lambda_n)\to \euE(\Lambda_1)$, with $\psi_n$ commuting with the edge-inversion maps, and there is $\eue_n\in \euE(\Lambda_n)$ with $v_n\ce o(\eue_n)$ such that:
    \begin{itemize}
        \item[(i)] $(\gog_n)_{v_n}$ is inaccessible, $\varphi_n(v_n)=\varphi_{n+1}(v_{n+1})$, and $(\gog_{n+1})_{\eue_{n+1}}$ is $G$-conjugate to a proper subgroup of $(\gog_n)_{\eue_n}$;
            \item[(ii)] for every $w\in \euV(\Lambda_n)\setminus\{v_n\}$, the group $(\gog_n)_w$ is $G$-conjugate to $(\gog_1)_{\varphi_n(w)}$ and, for every $\euf\in \euE(\Lambda_n)\setminus\{\eue_n,\overline{\eue}_n\}$, the group $(\gog_n)_{\euf}$ is $G$-conjugate to a subgroup of $(\gog_1)_{\psi_n(\euf)}$.
    \end{itemize}
    Then $\ccd_\Q(G)\neq 1$.
\end{prop}
\begin{proof}
    Assume that $\ccd_\Q(G)=1$ and let $n\geq 2$.
    By Proposition~\ref{prop: chiGrOfGr}, given an orientation $\euE_n^+$ of $\euE(\Lambda_n)$ with $\eue_n\in \euE_n^+$, one has
    \begin{equation}\label{eq:chiLambdan}
        \tilde{\chi}_G=\underbrace{\sum_{w\in \euV(\Lambda_n), \atop w\neq v_n}\tilde{\chi}_{(\gog_n)_w}-\sum_{\euf\in \euE_n^+, \atop \euf\neq\eue_n}\mu_{(\gog_n)_\euf}}_{=:\eta_n(G)}\,+\,\tilde{\chi}_{(\gog_n)_{v_n}}-\mu_{(\gog_n)_{\eue_n}}.
    \end{equation}
    By~(ii), one has $\tilde{\chi}_{(\gog_n)_{w}}=\tilde{\chi}_{(\gog_1)_{\varphi_n(w)}}$ for every $w\in \euV(\Lambda_n)\setminus\{v_n\}$. Moreover, again by (ii), for every $\euf\in \euE_n^+$ there is $g=g_\euf\in G$ such that $g(\gog_n)_{\euf}g^{-1}\leq (\gog_1)_{\psi_n(\euf)}$ and then
    \begin{equation*}
    \mu_{(\gog_n)_{\euf}}=\mu_{g(\gog_n)_{\euf}g^{-1}}=|(\gog_1)_{\psi_n(\euf)}:g(\gog_n)_{\euf}g^{-1}|\mu_{(\gog_1)_{\psi_n(\euf)}}\geq\mu_{(\gog_1)_{\psi_n(\euf)}}.
    \end{equation*}
   We now find a quantity $\eta(G)$, independent of $n$, such that 
   \begin{equation}\label{eq:chiLambdan2}
       \eta_n(G)-\mu_{(\gog_n)_{\eue_n}}\leq \eta(G)-\frac{1}{2}\mu_{(\gog_n)_{\eue_n}},\quad \forall\, n\geq 2.
   \end{equation}
    In detail, let $v\ce\varphi_2(v_2)$ and recall that, by (i), one has $v=\varphi_n(v_n)$ for every $n\geq 2$. Then, one observes what follows:
    \begin{equation*}
    \begin{split}
        \eta_n(G)-\mu_{(\gog_n)_{\eue_n}} & 
        \leq\sum_{w\in\euV(\Lambda_n), \atop w\neq v_n}\tilde{\chi}_{(\gog_n)_{w}}-\frac{1}{2}\cdot\sum_{\euf\in \euE_n^+,\atop \euf\neq \eue_n}\mu_{(\gog_n)_{\euf}}-\mu_{(\gog_n)_{\eue_n}}\\
        &=\sum_{w\in\euV(\Lambda_n), \atop w\neq v_n}\tilde{\chi}_{(\gog_n)_w}-\frac{1}{2}\cdot\sum_{\euf\in \euE_n^+}\mu_{(\gog_n)_{\euf}}-\frac{1}{2}\mu_{(\gog_n)_{\eue_n}}\\
            & \leq\sum_{w\in \euV(\Lambda_n), \atop w\neq v_n}\tilde{\chi}_{(\gog_1)_{\varphi_n(w)}}-\frac{1}{2}\cdot\sum_{\euf\in \euE_n^+}\mu_{(\gog_1)_{\psi_n(\euf)}}-\frac{1}{2}\mu_{(\gog_n)_{\eue_n}}\\
                 & =\underbrace{\sum_{w\in \euV(\Lambda_1), \atop w\neq v}\tilde{\chi}_{(\gog_1)_w}-\frac{1}{4}\cdot\sum_{\euf\in \euE(\Lambda_1)}\mu_{(\gog_1)_{\euf}}}_{=:\eta(G)}-\frac{1}{2}\mu_{(\gog_n)_{\eue_n}}.
    \end{split}
    \end{equation*}
   For the last equality of \eqref{eq:chiLambdan2} we have used that $\euE(\Lambda_1)=\psi_n(\euE(\Lambda_n))=\psi_n(\euE_n^+)\sqcup \overline{\psi_n(\euE_n^+)}$ and that $(\gog_1)_\euf=(\gog_1)_{\overline{\euf}}$ for every $\euf\in \euE(\Lambda_1)$.

Let $\{\caO_n\}_{n\geq 1}$ be the sequence of compact open subgroups of $G$ provided by (i) such that $\caO_1\geq \caO_2>\ldots >\caO_n>\ldots$ and $\caO_n$ is $G$-conjugate to $(\gog_n)_{\eue_n}$ for each $n\geq 1$. Then, for every $n\geq 2$,
    \begin{gather}\label{eq:On}
\mu_{(\gog_n)_{\eue_n}}=\mu_{\caO_n}=|\caO_1:\caO_n|\mu_{\caO_1}=\Bigg(\prod_{i=1}^{n-1}|\caO_i:\caO_{i+1}|\Bigg)\mu_{\caO_1}\geq 2^{n-1}\mu_{\caO_1}.
    \end{gather}
Moreover, Theorem~\ref{thm:chi} yields $\tilde{\chi}_{(\gog_n)_{\varphi_n(v)}}\leq 0$. This, together with \eqref{eq:chiLambdan}, \eqref{eq:chiLambdan2} and \eqref{eq:On}, implies that
   \begin{equation*}
       \tilde{\chi}_G\leq \eta_n(G)-\mu_{(\gog_n)_{\eue_n}}=\eta_n(G)-\mu_{\caO_n}\leq \eta(G)-2^{n-2}\mu_{\caO_1},
   \end{equation*}
    for every $n\geq 2$. This contradicts the fact that $\tilde\chi_G=r\cdot \mu_{\caO_1}$ for some fixed $r\in\Q$.
\end{proof}

\begin{lem}\label{lem:inacc}
    Let $G$ be a compactly generated inaccessible t.d.l.c.~group. Consider a finite proper graph of t.d.l.c.~groups $(\gog,\Lambda)$ with compact edge-groups such that $G\simeq\pi_1(\gog,\Lambda)$, and let $v\in \euV(\Lambda)$ such that $\gog_v$ is inaccessible. Then there are two finite graphs of t.d.l.c.~groups $(\euH,\Gamma)$ and $(\gog',\Lambda')$ with compact edge-groups such that: $(\gog',\Lambda')$ is proper, $|\euE(\Gamma)|=|\euE(\Lambda)|+2$, $|\euE(\Lambda')|\in\{|\euE(\Lambda)|, |\euE(\Lambda)|+2\}$, and $G\simeq \pi_1(\euH,\Gamma)\simeq \pi_1(\gog',\Lambda')$.
    
    In addition, the pair $((\euH,\Gamma), (\gog',\Lambda'))$ satisfies exactly one of the two following conditions:
    \begin{itemize}
        \item[(A)] $(\euH,\Gamma)$ is proper, $(\gog',\Lambda')=(\euH,\Gamma)$ and there is an injective map $\iota\colon \euV(\Lambda)\hookrightarrow \euV(\Lambda')$ such that $\gog'_{\iota(v)}$ is $G$-conjugate to an inaccessible proper subgroup of $\gog_v$;
        \item[(B)] $(\euH,\Gamma)$ is not proper. Moreover, there are bijections $$\varphi\colon \euV(\Lambda')\to \euV(\Lambda)\quad\text{and}\quad\psi\colon \euE(\Lambda')\to \euE(\Lambda),$$ with $\psi$ commuting with the edge-inversion maps, such that the following properties hold:   
    \begin{itemize}
        \item[(B1)] $\gog'_{\varphi^{-1}(v)}$ is an inaccessible proper subgroup of $\gog_{v}$;
        \item[(B2)] for every $w\in \euV(\Lambda')\setminus\{\varphi^{-1}(v)\}$, the group $\gog'_w$ is $G$-conjugate to $\gog_{\varphi(w)}$ and, for every $\euf'\in \euE(\Lambda')$, the group $\gog_{\euf'}$ is $G$-conjugate to a subgroup of $\gog_{\psi(\euf')}$. In particular, there is $\eue\in \euE(\Lambda')$ with $o(\eue)=\varphi^{-1}(v)$ such that $\gog'_{\eue}$ is $G$-conjugate to a proper subgroup of $\gog_{\psi(\eue)}$. 
    \end{itemize}
    \end{itemize}
\end{lem}

\begin{proof}
    By Proposition~\ref{prop: AccFact}, there is a vertex-group $\gog_v$ of $(\gog,\Lambda)$ which is inaccessible. By Fact~\ref{fact:one end}, $\gog_v$ splits non-trivially over a compact open subgroup~$C$, say either $\gog_v=A\ast_{C}B$ or $\gog_v=A\ast_C^t$. In other words, $\gog_v$ is isomorphic to the fundamental group of either a $1$-segment of groups (if $\gog_v=A\ast_CB$):
\begin{equation*}
        \xymatrix{\stackrel{A}{\bullet}_v\ar@{-}[rr]^{C}_{\eue}&&\stackrel{B}{\bullet}_w}
    \end{equation*}
     or a $1$-loop of groups (if $\gog_v=A\ast_C^t$):
    \begin{equation*}
    \xymatrix{\stackrel{A}{\bullet}_v\ar@{-}^{C}_{\eue}@(ur,dr)}
    \end{equation*}
    Without loss of generality, assume that $A$ is inaccessible.
For every $\euf\in \euE(\Lambda)$ with~$t_\Lambda(\euf)=v$ (here $t_\Lambda(\argu)$ is the terminus map in~$\Lambda$), assign~$\euf$ a pair~$(g_{\euf},x_{\euf})$ with $g_{\euf}\in \gog_v$ and $x_{\euf}\in\{v,w\}$ satisfy $\alpha_{\euf}(\gog_{\euf})\subseteq g_{\euf}Ag_{\euf}^{-1}$ if $x_{\euf}=v$ and $\alpha_{\euf}(\gog_{\euf})\subseteq g_{\euf}Bg_{\euf}^{-1}$ if $x_{\euf}=w$, provided $\alpha_{\euf}\colon \gog_{\euf}\hookrightarrow \gog_v$ is the monomorphism associated to~$\euf$ in~$(\gog,\Lambda)$. By the Bruhat--Tits fixed point theorem, the pair $(g_{\euf},x_{\euf})$ exists. 

Assume first that $\gog_v=A\ast_CB$. Let $(\euH,\Gamma)$ be the graph of t.d.l.c.~groups obtained from~$(\gog,\Lambda)$ by expanding $\{v\}$ with $\xymatrix{\stackrel{v}{\bullet}\ar@{-}[r]^{\eue}&\stackrel{w}{\bullet}}$ with respect to $\{(g_{\euf},x_{\euf})\}_{\euf\in \euE(\Lambda),\,t(\euf)=v}$ and such that $G\simeq \pi_1(\euH,\Gamma)$ (cf.~the construction before Lemma~\ref{lem:exp}). Note that every edge~$\euf$ of~$\Lambda$ ending in~$v$ will terminate, as an edge of~$\Gamma$, either in~$v$ or in~$w$ according to $x_{\euf}$ .
Note also that $\euH_v=A$, $\euH_w=B$ and $\euH_\eue=C$. 

Analogously, if $\gog_v=A\ast_C^t$, let $(\euH,\Gamma)$ be the  graph of t.d.l.c.~groups obtained from~$(\gog,\Lambda)$  by expanding $\{v\}$ with $\xymatrix{\stackrel{v}{\bullet}\ar@{-}^{\eue}@(ur,dr)}$ with respect to \linebreak $\{(g_\euf, v)\}_{\euf\in\euE(\Lambda),\,t(\euf)=v}$ and such that $G\simeq\pi_1(\euH,\Gamma)$. 
Note that $\euH_v=A$ and~$\euH_\eue=C$.

In both cases, one has $|\euE(\Gamma)|=|\euE(\Lambda)|+2$ and the vertex-group $\euH_v=A$ is (by assumption) an inaccessible proper subgroup of $\gog_v$.

If $(\euH,\Gamma)$ is proper, for $(\gog',\Lambda'):=(\euH,\Gamma)$ one easily checks that (A) holds.
Assume that $(\euH,\Gamma)$ is not proper. Since $(\gog,\Lambda)$ is proper and $\gog_v$ is non-compact, necessarily $\gog_v=A\ast_C B$ and there is a path $(\eue,\euf)$ in $\Gamma$ with $\euf\neq \bar\eue,$ because $C$ is properly contained in $B$, and $w=o(\euf)\neq t(\euf)=w'$ such that $\euH_w=\euH_{o(\euf)}=\euH_\euf\subsetneq\euH_{t(\euf)}=\euH_{w'}$.
        \begin{align*}
\xymatrix{
\stackrel{\euH_{v}}{\bullet}\ar@{-}[rr]_-{\euH_{\eue}}&&\stackrel{\euH_{w}=\euH_{\euf}}{\bullet}\ar@{-}[rr]_-{\euH_{\euf}}&&\stackrel{\euH_{t(\euf)}=\euH_{w'}}{\bullet}}
\end{align*}
Let $(\gog',\Lambda')$ be the finite graph of t.d.l.c.~groups obtained from $(\euH,\Gamma)$ by contracting the edge $\euf$, i.e.,
$\euV(\Lambda'):=\euV(\Gamma)\setminus\{w\}=\euV(\Lambda)$, $\euE(\Lambda'):=\euE(\Gamma)\setminus\{\euf,\overline{\euf}\}$, and the origin and the terminus of $\eue$ 
 in $\Lambda'$ are the vertices $v$ and $w'$, respectively. Notice that all edges of $\Lambda$ with origin $w$ that survive in $\Lambda'$ have now origin in $w'$ after the relabelling.
 $$
 \begin{array}{lcl}
    {\xymatrix@C=1em{\bullet\ar@/_2.0pc/@{--}[dd]& &\bullet_{w'}&&\bullet\ar@/^2.0pc/@{--}[dd]\\
 &\bullet_v\ar@{-}[lu]\ar@{-}[ld]\ar@{-}[rr]_{\eue}&&\bullet_w\ar@{-}[lu]^{\euf}\ar@{-}[ru]^{\mathbf{a}}\ar@{-}[rd]_{\mathbf{b}}&\\
 \bullet&&&&\bullet}} & \quad\to\quad & {\xymatrix@C=1em{\bullet\ar@/_2.0pc/@{--}[dd]& &&&\bullet\ar@/^2.0pc/@{--}[dd]\\
 &\bullet_v\ar@{-}[lu]\ar@{-}[ld]\ar@{-}[rr]_{\eue}&&\bullet_{w'}\ar@{-}[ru]^{\mathbf{a}}\ar@{-}[rd]_{\mathbf{b}}&\\
 \bullet&&&&\bullet}}
 \end{array}
 $$
 Moreover, $\gog'_{x}=\euH_x$ for all $x\in \euV(\Lambda')\sqcup \euE(\Lambda')$. Note that $(\gog',\Lambda')$ is proper. Indeed, since $(\gog,\Lambda)$ is proper, the only possibility for $(\gog',\Lambda')$ not being proper is that there exists $\mathbf{a}\in \euE(\Lambda')\setminus\{\overline{\eue}\}$ with $o(\mathbf{a})={w'}$ such that $\gog'_{w'}=\gog'_{\mathbf{a}}$. However, as $\mathbf{a}$ was an edge in $\Lambda$ with origin $w$, one has that $\euH_{\mathbf{a}}\subset\euH_{w}$. Moreover, the inclusion must be strict because $(\gog,\Lambda)$ is proper. Therefore, $\gog'_{\mathbf{a}}=\euH_{\mathbf{a}}\subsetneq \euH_{w}=\euH_\euf\subsetneq\euH_{t(\euf)}=\gog'_{w'}$. 
 Define $\varphi:=\iid_{\euV(\Lambda)}$ and  
$$\psi\colon\euE(\Lambda')\to\euE(\Lambda),\quad\psi(\euf'):= \begin{cases}
     \euf',\quad\text{if $\euf'\in\euE(\Lambda')\setminus\{\eue,\overline{\eue}\}$,}\\
     \euf,\quad\text{if $\euf'=\eue,$}\\
     \bar{\euf},\quad\text{if $\euf'=\bar{\eue}$.}
 \end{cases}$$
It is now easy to check that the pair $((\euH,\Gamma),(\gog',\Lambda'))$ satisfies condition (B) of the claim.
\end{proof}
\begin{prop}\label{prop:inacc}
    Let $G$ be a compactly generated inaccessible t.d.l.c.~group. Then there exists an infinite sequence $\{(\gog_n,\Lambda_n)\}_{n\geq 1}$ of finite proper graphs of t.d.l.c.~groups having compact edge-groups satisfying $G\simeq \pi_1(\gog_n,\Lambda_n)$ for every $n\geq 1$. Moreover, exactly one of the following conditions holds:
    \begin{itemize}
        \item[(a)] $|\euE(\Lambda_n)|=2n$ for every $n\geq 1$;
        \item[(b)] for every $n\geq 2$, there are bijections $\varphi_n\colon \euV(\Lambda_n)\to\euV(\Lambda_1)$ and $\psi_n\colon \euE(\Lambda_n)\to \euE(\Lambda_1)$, with $\psi_n$ commuting with the edge-inversion maps, and there is $\eue_n\in \euE(\Lambda_n)$ with $v_n\ce o(\eue_n)$ such that:
        \begin{itemize}
            \item[(i)] $(\gog_n)_{v_n}$ is inaccessible;
            \item[(ii)] for every $w\in \euV(\Lambda_n)\setminus\{v_n\}$, the group $(\gog_n)_w$ is $G$-conjugate to $(\gog_1)_{\varphi_n(w)}$ and, for every $\euf\in \euE(\Lambda_n)$, the group $(\gog_n)_{\euf}$ is $G$-conjugate to a subgroup of $(\gog_1)_{\psi_n(\euf)}$.
        \end{itemize}
        Moreover, for every $n\geq 2$, one has $\varphi_n(v_n)=\varphi_{n+1}(v_{n+1})$ and $(\gog_{n+1})_{\eue_{n+1}}$ is $G$-conjugate to a proper subgroup of $(\gog_n)_{\eue_n}$.
    \end{itemize}
\end{prop}
\begin{proof}
  By Fact~\ref{fact:one end}, there is a finite proper graph of t.d.l.c.~groups $(\gog_1,\Lambda_1)$ with compact edge-groups such that $|\euE(\Lambda_1)|=2$ and $G\simeq\pi_1(\gog_1,\Lambda_1)$.
  Applying Lemma~\ref{lem:inacc} and proceeding by induction, one constructs two infinite sequences of finite graphs of t.d.l.c.~groups $\{(\gog_n,\Lambda_n)\}_{n\geq 1}$ and $\{(\euH_n,\Gamma_n)\}_{n\geq 1}$ with compact edge-groups satisfying, for every $n\geq 1$, what follows: $(\gog_n,\Lambda_n)$ is proper, 
 $|\euE(\Gamma_{n+1})|=|\euE(\Lambda_n)|+2$, $|\euE(\Lambda_{n+1})|\in\{|\euE(\Lambda_n)|, |\euE(\Lambda_n)|+2\}$, $G\simeq\pi_1(\gog_n,\Lambda_n)\simeq \pi_1(\euH_n,\Gamma_n)$, and the pair $((\euH_n,\Gamma_n), (\gog_{n+1},\Lambda_{n+1}))$ meets either condition~(A) or condition~(B) of Lemma~\ref{lem:inacc}. 
  If there are infinitely many $n\geq 1$ for which condition~(A) is satisfied, there is a subsequence of $\{(\gog_n,\Lambda_n)\}_{n\geq 1}$ satisfying~(a). 
  
  Assume instead that there is $n_0\geq 1$ such that condition (B) holds for every $n\geq n_0$. Let $v\in \euV(\Lambda_{n_0})$ such that $(\gog_{n_0})_{v}$ is inaccessible (cf.~Proposition~\ref{prop: AccFact}).
  By Lemma~\ref{lem:inacc}, for all $n\geq n_0$ there are bijections $\tilde{\varphi}_{n+1}\colon\euV(\Lambda_{n+1})\to \euV(\Lambda_n)$ and $\tilde{\psi}_{n+1}\colon \euE(\Lambda_{n+1})\to \euE(\Lambda_n)$, with $\tilde{\psi}_{n+1}$ commuting with the edge-inversion maps, and there is $v_{n+1}\in \euV(\Lambda_{n+1})$ such that $\tilde{\varphi}_{n_0+1}(v_{n_0+1})=v$ and $\tilde{\varphi}_{n+1}(v_{n+1})=v_n$ for every $n\geq n_0+1$. Moreover, for every $n\geq n_0$, the following holds:
    \begin{itemize}
        \item[(b1)] $(\gog_{n+1})_{v_{n+1}}$ is inaccessible;
        \item[(b2)] for every $w\in \euV(\Lambda_{n+1})\setminus\{v_{n+1}\}$, the group $(\gog_{n+1})_w$ is $G$-conjugate to $(\gog_n)_{\tilde{\varphi}_{n+1}(w)}$ and, for every $\euf\in \euE(\Lambda_{n+1})$, the group $(\gog_{n+1})_\euf$ is $G$-conjugate to a subgroup of $(\gog_n)_{\tilde{\psi}_{n+1}(\euf)}$. In particular, there is $\eue\in \euE(\Lambda_{n+1})$ with $o(\eue)=v_{n+1}$ such that $(\gog_{n+1})_{\eue}$ is $G$-conjugate to a proper subgroup of $(\gog_n)_{\tilde{\psi}_{n+1}(\eue)}$. 
    \end{itemize}
  Up to replacing $\{(\gog_n,\Lambda_n)\}_{n\geq n_0}$ with a subsequence, one may assume that for every $n\geq n_0$ there is $\eue_n\in \euE(\Lambda_n)$ with $o(\eue_n)=v_n$ such that $(\gog_{n+1})_{\eue_{n+1}}$ is $G$-conjugate to a proper subgroup of $(\gog_n)_{\eue_n}$. This is indeed a consequence of (b2) and the fact that $\Lambda_{n_0}$ has finitely many edges. 
  Defining $\varphi_n\ce\tilde{\varphi}_n\circ\ldots\circ \tilde{\varphi}_{n_0+1}$ and $\psi_n\ce\tilde{\psi}_n\circ\ldots \circ\tilde{\psi}_{n_0+1}$ for every $n>n_0$, one concludes that $\{(\gog_n,\Lambda_n)\}_{n\geq n_0}$ satisfies condition (b) of the claim. 
\end{proof}
Combining Proposition~\ref{prop:inacc} and Proposition~\ref{prop:deg}, one deduces the following.
\begin{cor}\label{cor:inaccuni}
    Let $G$ be a compactly generated unimodular t.d.l.c.~group with $\ccd_\Q(G)=1$. If $G$ is inaccessible, then there is an infinite sequence $\{(\gog_n,\Lambda_n)\}_{n\geq 1}$ of finite proper graphs of t.d.l.c.~groups having compact edge-groups, with $|\euE(\Lambda_n)|=2n$ and $G\simeq \pi_1(\gog_n,\Lambda_n)$ for every $n\geq 1$.
\end{cor}
\subsection{$\CO$-bounded t.d.l.c.~groups}\label{sus:CObound}
Let $\mu$ be a left-invariant Haar measure on a t.d.l.c.~group  $G$. The group $G$ is said to be \emph{$\mathcal{CO}$-bounded} if 
\begin{equation}
\nor G\nor_\mu\ce\sup \{\mu(U): U\in \mathcal{CO}(G)\}<\infty.
\end{equation}
For instance, a discrete group is $\CO$-bounded if there is a bound on the order of its finite subgroups.
Being $\CO$-bounded does not depend on the choice of the Haar measure. Indeed, if $\lambda$ is another left-invariant Haar measure on $G$, there is $c>0$ such that $\lambda=c\cdot\mu$. 
\begin{prop}\label{prop:COunimod}
Every $\CO$-bounded t.d.l.c.~group is unimodular.
\end{prop}
\begin{proof}
Let $\Delta\colon G\to\Q_{>0}$ denote the modular function of $G$ and, for a given $U\in \mathcal{CO}(G)$, let $\mu_U$ be the left-invariant Haar measure on $G$ such that $\mu_U(U)=1$. Observe that $\Delta(g)=\mu_U(g^{-1}Ug)\leq \nor G\nor_{\mu_U}$, for every $g\in G$.
    Since $\Q_{>0}$ has no non-trivial bounded subgroups, $G$ is unimodular.
\end{proof}
\begin{prop}
\label{prop: COnoeth}
Let $G$ be a $\CO$-bounded t.d.l.c.~group. 
 Every strictly increasing sequence of compact open subgroups of $G$ is finite. 
 \end{prop}
    \begin{proof}
     Every strictly ascending chain of compact open subgroups $U_1<U_2<\cdots<U_n<\cdots$ yields the following bounded (and hence finite) strictly increasing sequence of positive integers
$$1<\mu_{U_1}(U_2)=|U_2:U_1|<\cdots<\mu_{U_1}(U_n)=|U_n:U_1|<\cdots.$$
     Indeed, for every $i\geq 2$, the set $ U_i\setminus U_{i-1}$ is a non-empty open set in $G$ and, therefore,  $0<\mu_{U_1}(U_i\setminus U_{i-1})=\mu_{U_1}(U_i)-\mu_{U_1}(U_{i-1}).$
\end{proof}
\begin{ex}\label{ex: prelThm}
 Let $G$ be a t.d.l.c.~group.
\begin{enumerate}
  \item[(a)] If $G$ is unimodular and admits an $\underbar{E}_{\CO}(G)$-space with finitely many orbits on the $0$-skeleton, then $G$ is $\CO$-bounded. Indeed, every $\caO\in\CO(G)$ stabilises a point of $X$ and henceforth it is subconjugated to a vertex-stabiliser (cf.~\cite[\S6.6]{cw:qrat}). By unimodularity and since there are finitely many conjugacy classes of stabilisers of vertices, $G$ is $\CO$-bounded. 
    
    For instance, if $G$ is unimodular and $G\simeq \pi_1(\gog,\Lambda)$ for some graph $(\gog,\Lambda)$ of profinite groups with $|\euV(\Lambda)|<\infty$, then $G$ is $\CO$-bounded. Indeed, the universal tree of $(\gog,\Lambda)$ is a $1$-dimensional model for $\underbar{E}_\CO(G)$ (cf.~\cite[Theorem~4.9]{lu}) and has finitely many orbits on the $0$-skeleton.
    In particular, if $\ccd_\Q(G)\leq 1$ and $G$ is either finitely generated and discrete or {unimodular and} compactly presented, then $G$ is $\CO$-bounded (cf.~Theorem~\hyperref[thmA]{A}~and~Theorem~\hyperref[thmC]{C}).
     
   \item[(b)] One can introduce the Bredon cohomology for $G$ with respect to the family of all compact open subgroups (cf.~\cite[p.~14ff]{mv}). A natural generalisation of (a) is given by unimodular t.d.l.c.~groups which are of {\it type Bredon $\FP_0$}, i.e., there is a finite subset $\mathcal A$ of $\CO(G)$ such that every $\caO\in\CO(G)$ is subconjugate to some element of $\mathcal A$ (cf.~\cite[\S3]{kmpn}).
\end{enumerate}
\end{ex}
\begin{thm}\label{thm: firstRes}
Let $G$ be a compactly generated t.d.l.c.~group satisfying $\ccd_\Q(G)=1$. If $G$ is $\CO$-bounded, then $G$ is accessible.
\end{thm}
\begin{proof}
Suppose that $G$ is inaccessible. 
Let $\{(\gog_n,\Lambda_n)\}_{n\geq 1}$ be the sequence produced by
Corollary~\ref{cor:inaccuni}. For every $n\geq 1$ fix an edge orientation $\euE_n^+\subseteq\euE(\Lambda_n)$, a maximal subtree $\euT_n$ of $\Lambda_n$ and $v_n\in\euV(\Lambda_n)$ such that $(\gog_n)_{v_n}$ is non-compact (such a vertex exists because $G$ is inaccessible).
By \cite[\S I.2, Proposition~12]{ser:trees}, there exists a bijection $\euE_n^+\cap\euE(\euT_n)\longrightarrow \euV(\euT_n)\setminus \{v_n\}=\euV(\Lambda_n)\setminus \{v_n\}$. Replacing $\eue\in\euE_n^+\cap \euE(\euT_n)$ by $\bar{\eue}$ if necessary, we may assume that the bijection is given by the origin map $\eue\mapsto o(\eue)$. By Proposition~\ref{prop: chiGrOfGr}, for every $n\geq 1$ one has 
\begin{equation}\label{eq:chiG}
\begin{split}
\tilde\chi_G & =  \sum_{v\in \euV(\Lambda_n)}\tilde\chi_{(\gog_n)_v}-\sum_{\eue\in \euE_n^+}\mu_{(\gog_n)_\eue}\\
&=\tilde\chi_{(\gog_n)_{v_n}}+\sum_{\eue\in \euE_n^+\cap \euE(\euT_n)} (\tilde\chi_{(\gog_n)_{o(\eue)}}-\mu_{(\gog_n)_\eue})-\sum_{\eue\in \euE_n^+\setminus \euE(\euT_n)} \mu_{(\gog_n)_\eue}\\
    &\leq \sum_{\eue\in \euE_n^+\cap \euE(\euT_n)} (\tilde\chi_{(\gog_n)_{o(\eue)}}-\mu_{(\gog_n)_\eue})-\sum_{\eue\in \euE_n^+\setminus \euE(\euT_n)} \mu_{(\gog_n)_\eue},
    \end{split}
\end{equation}
where the last inequality holds since $\tilde\chi_{(\gog_n)_{v_n}}\leq 0$ (cf.~Theorem~\ref{thm:chi}). Let $\mu$ be an arbitrary left-invariant Haar measure on $G$.
By Corollary~\ref{cor: 1stEstimateChar}, for every $\eue\in \euE_n^+\cap \euE(\euT_n)$ one has
\begin{equation}\label{eq:claim}
    \tilde\chi_{(\gog_n)_{o(\eue)}}-\mu_{(\gog_n)_\eue}\leq -\frac{1}{2}\mu_{(\gog_n)_\eue}=-\frac{1}{2\mu((\gog_n)_{\eue})}\mu\leq-\frac{1}{2\nor G\nor_\mu}\mu.
\end{equation}
Moreover, for every $\eue\in \euE_n^+\setminus \euE(\euT_n)$ one observes that
\begin{equation}\label{eq:claim2}
    \mu_{(\gog_n)_\eue}=\frac{1}{\mu((\gog_n)_\eue)}\mu\geq \frac{1}{\nor G\nor_\mu}\mu>\frac{1}{2\nor G\nor_\mu}\mu.
\end{equation}
 Hence, from~\eqref{eq:chiG}, \eqref{eq:claim} and \eqref{eq:claim2} one concludes that
\begin{alignat*}{2}
    \tilde\chi_G &\leq \Bigg(-\frac{|\euE_n^+\cap \euE(\euT_n)|}{2\nor G\nor_\mu}-\frac{|\euE_n^+\setminus \euE(\euT_n)|}{2\nor G\nor_\mu}\Bigg)\cdot\mu\\
       & =-\frac{|\euE_n^+|}{2\nor G\nor_\mu}\mu=-\frac{|\euE(\Lambda_n)|}{4\nor G\nor_\mu}\mu\quad\text{for every}\ n\geq 1,
\end{alignat*}
a contradiction. Therefore, $G$ is accessible. 
\end{proof}
\begin{rem}
  A potential generalisation of the Dunwoody's foldings (cf.~\cite[Theorem~2.3]{dun:fold}) in the context of t.d.l.c.~groups might lead to an alternative proof of Theorem~\ref{thm: firstRes}. Nevertheless, in order to prove that a t.d.l.c.~group is accessible, one would still need to exclude t.d.l.c.~groups admitting an infinite sequence of decompositions as considered in Proposition~\ref{prop:inacc}(b).
\end{rem}
Finally, Theorem \ref{thm: firstRes} yields the following characterisation.
\begin{cor}\label{cor:acc}
For a compactly generated t.d.l.c.~group $G$, the following are equivalent:
\begin{enumerate}
    \item[(a)] $G$ is $\CO$-bounded and $\ccd_\Q(G)\leq 1$;
    \item[(b)] $G$ is unimodular and $G\simeq \pi_1(\gog,\Lambda)$, for some finite graph $(\gog, \Lambda)$ of profinite groups;
    \item[(c)] $G$ is unimodular and  some (and hence every) Cayley--Abels graph of $G$ is quasi-isometric to a tree;
    \item[(d)] $G$ has a finitely generated free subgroup which is cocompact and discrete.
\end{enumerate}
\end{cor}
\begin{proof}
The implication (a)$\Rightarrow$(b) is Proposition~\ref{prop:COunimod}, Theorem~\ref{thm: firstRes} and Fact~\ref{fact:one end}. The converse is Fact~\ref{fsct:vice} and Example~\ref{ex: prelThm}(a). Finally, \cite[Theorem~3.28]{km:rough} yields (b)$\Leftrightarrow$(c)$\Leftrightarrow$(d).
\end{proof}
\section*{Acknowledgements}
We thank the referee for all the comments that helped us to improve our manuscript. The first and second named authors were supported by the Deutsche Forschungsgemeinschaft (DFG, German Research Foundation) – Project-ID 491392403 – TRR 358. The third author gratefully acknowledges financial support by the PRIN2022 “Group theory and its applications”.
All authors are members of the Gruppo Nazionale per le Strutture Algebriche,
Geometriche e le loro Applicazioni (GNSAGA), which is part of the Istituto
Nazionale di Alta Matematica (INdAM). The authors gratefully acknowledge also the hospitality of the Mathematisches Forschungsinstitut Oberwolfach where they met in November 2023 and had the occasion to work on this paper.

\section*{Data availability}
Data sharing is not applicable to this article as no data sets were generated or analysed during the current study.

\section*{Conflict of interest}
On behalf of all authors, the corresponding author states that there is no conflict of interest.

\bibliography{TDLC-2}
\bibliographystyle{amsplain}
\end{document}